\documentclass [twoside, reqno, 12pt] {amsart}

\usepackage{hyperref}

\usepackage{amsfonts}
\usepackage{amssymb}
\usepackage{a4}
\usepackage{color}
\usepackage{amscd,amssymb,amsfonts}

\usepackage{picture,xcolor}

\usepackage{etoolbox}

\usepackage{tikz}

\newtheorem{thm}{Theorem}[section]
\newtheorem{cor}[thm]{Corollary}
\newtheorem{lem}[thm]{Lemma}
\newtheorem{prop}[thm]{Proposition}

\newtheorem{rem}[thm]{Remark}

\theoremstyle{definition}

\numberwithin{equation}{section}

\newcommand{\C}{\mathbb{C}}

\newcommand{\N}{\mathbb{N}}

\newcommand{\R}{\mathbb{R}}

\def\hat{\widehat}
\def\tilde{\widetilde}
\def \bfo {\begin {eqnarray*} }
\def \efo {\end {eqnarray*} }
\def \ba {\begin {eqnarray*} }
\def \ea {\end {eqnarray*} }
\def \beq {\begin {eqnarray}}
\def \eeq {\end {eqnarray}}

\def \p {\partial}

\def\hat{\widehat}
\def\tilde{\widetilde}
\def \bfo {\begin {eqnarray*} }
\def \efo {\end {eqnarray*} }
\def \ba {\begin {eqnarray*} }
\def \ea {\end {eqnarray*} }
\def \beq {\begin {eqnarray}}
\def \eeq {\end {eqnarray}}

\def \p {\partial}

\begin{document}

\title[Inverse problems for nonlinear biharmonic operators]{Inverse problems for third--order nonlinear perturbations of  biharmonic operators}

\author[Bhattacharyya]{Sombuddha Bhattacharyya}

\address
       {S. Bhattacharyya, Department of Mathematics\\
       Indian Institute of Science Education and Research\\ 
       Bhopal, India (IISERB)}
\email{sombuddha@iiserb.ac.in}

\author[Krupchyk]{Katya Krupchyk}
\address
        {K. Krupchyk, Department of Mathematics\\
University of California, Irvine\\
CA 92697-3875, USA }

\email{katya.krupchyk@uci.edu}

\author[Kumar Sahoo]{Suman Kumar Sahoo}

\address
       {S. K. Sahoo, Seminar for Applied Mathematics\\
        Department of Mathematics\\ 
        ETH Z\"urich, Switzerland}
\email{susahoo@ethz.ch}

\author[Uhlmann]{Gunther Uhlmann}

\address
       {G. Uhlmann, Department of Mathematics\\
       University of Washington\\
       Seattle, WA  98195-4350\\
       USA\\
        and Institute for Advanced Study of the Hong Kong University of Science and Technology}
\email{gunther@math.washington.edu}

\maketitle

\begin{abstract}
We study inverse boundary problems for third--order nonlinear tensorial perturbations of biharmonic operators on a bounded domain in $\mathbb{R}^n$, where $n\geq 3$. By imposing appropriate assumptions on the nonlinearity, we demonstrate that the Dirichlet--to--Neumann map, known on the boundary of the domain, uniquely determines the genuinely nonlinear tensorial third-order perturbations of the biharmonic operator. The proof relies on the inversion of certain generalized momentum ray transforms on symmetric tensor fields. 
Notably, the corresponding inverse boundary problem for linear tensorial third-order perturbations of the biharmonic operator remains an open question.
\end{abstract}

\section{Introduction and statement of results}
Let $\Omega\subset \R^n$, $n\ge 3$, be a bounded open set with $C^\infty$ boundary. Consider the semilinear partial differential operator, 
\begin{equation}
\label{int_eq_1}
\begin{aligned}
L_{A^{(1)}, A^{(2)}, A^{(3)},q}u:= (-\Delta)^2u+ \sum_{i_1,i_2, i_3=1}^n A^{(3)}_{i_1 i_2 i_3}(x,u)D^3_{i_1i_2 i_3}u +\sum_{i_1,i_2=1}^n A^{(2)}_{i_1i_2}(x,u)D^2_{i_1i_2}u\\
+\sum_{i_1=1}^n A^{(1)}_{i_1}(x,u)D_{i_1}u+q(x,u),
\end{aligned}
\end{equation}
where $u\in C^\infty(\overline{\Omega})$ and $D^j_{i_1\dots i_j}=\frac{1}{i^j}\frac{\p^j}{\p_{x_{i_1}}\dots \p_{x_{i_j}}}$, $i_1, \dots, i_j\in \{1,2,\dots, n\}$, $j=1,2,3$. We let $S^j:=S^j(\R^n)$ stand for the space of symmetric $j$--tensors on $\R^n$. Let $0<\alpha<1$, and let $C^{0,\alpha}(\overline{\Omega}; S^j)$ stand for the space of H\"older continuous symmetric $j$--tensor fields, $j=0,1,2,3$. We assume that 
$A^{(j)}: \overline{\Omega}\times \C\to S^j$, $j=1,2,3$, and $q:\overline{\Omega}\times\C\to \C$ satisfy the following conditions, 
\begin{itemize}
\item[(i)] the map $\C\ni z\mapsto A^{(j)}(\cdot,z)$ is holomorphic with values in $C^{0,\alpha}(\overline{\Omega};  S^j)$ for some $0<\alpha<1$,

\item[(ii)] the map $\C\ni z\mapsto q(\cdot,z)$ is holomorphic with values in $C^{0,\alpha}(\overline{\Omega})$,

\item[(iii)] $A^{(j)}(x,0)=0$, $q(x,0)=0$, and $\p_z q(x,0)=0$ for all $x\in \overline{\Omega}$, $j=1,2,3$. 
\end{itemize}

It follows from (i), (ii), (iii) that $A^{(j)}$, $j=1,2,3$, and $q$ can be expanded into power series 
\begin{equation}
\label{int_eq_2}
A^{(j)}(x,z)=\sum_{k=1}^\infty A^{(j),k}(x)\frac{z^k}{k!}, \quad A^{(j),k}(x):=\p_z^kA^{(j)}(x,0)\in C^{0,\alpha}(\overline{\Omega};  S^j),
\end{equation}
converging in the $C^{0,\alpha}(\overline{\Omega};  S^j)$ topology, and 
\begin{equation}
\label{int_eq_3}
q(x,z)=\sum_{k=2}^\infty q^{k}(x)\frac{z^k}{k!}, \quad q^{k}(x):=\p_z^kq(x,0)\in C^{0,\alpha}(\overline{\Omega}),
\end{equation}
converging in the $C^{0,\alpha}(\overline{\Omega})$ topology. 

Consider the Dirichlet problem for the operator $L_{A^{(1)}, A^{(2)}, A^{(3)},q}$, 
\begin{equation}
\label{int_eq_4}
\begin{cases}
L_{A^{(1)}, A^{(2)}, A^{(3)},q}u=0\quad \text{in}\quad \Omega,\\
u|_{\p \Omega}=f,\\
\p_\nu u|_{\p \Omega}=g,
\end{cases}
\end{equation}
where $\nu$ is the unit outer normal to the boundary $\p \Omega$. It is shown in Theorem \ref{thm_well-posedness} below that under the assumptions (i), (ii), (iii), there exist $\delta>0$ and $C>0$ such that when $(f,g)\in B_\delta(\p \Omega):=\{(f,g)\in C^{4,\alpha}(\p \Omega)\times C^{3,\alpha}(\p \Omega): \|f\|_{C^{4,\alpha}(\p \Omega)}+\|g\|_{C^{3,\alpha}(\p \Omega)}<\delta\}$, the problem \eqref{int_eq_4} has a unique solution $u=u_{f,g}\in C^{4,\alpha}(\overline{\Omega})$ satisfying $\|u\|_{C^{4,\alpha}(\overline{\Omega})}<C\delta$.  Associated to the problem \eqref{int_eq_4}, we define the Dirichlet--to--Neumann map
\begin{equation}
\label{int_eq_5}
\Lambda_{A^{(1)}, A^{(2)}, A^{(3)},q}(f,g)= (\p^2_\nu u |_{\p \Omega}, \p^3_\nu u|_{\p \Omega}), 
\end{equation}
where $(f,g)\in B_\delta(\p \Omega)$ and $u=u_{f,g}$. The inverse problem that we are interested in is whether the knowledge of the Dirichlet--to--Neumann map $\Lambda_{A^{(1)}, A^{(2)}, A^{(3)},q}$ determines the nonlinear tensorial potentials $A^{(1)}$, $A^{(2)}$, $A^{(3)}$ and $q$ in the operator $L_{A^{(1)}, A^{(2)}, A^{(3)},q}$ given by \eqref{int_eq_1}, uniquely. 

Our main result is as follows. 
\begin{thm}
\label{thm_main}
Let $\Omega\subset \R^n$, $n\ge 3$, be a bounded open set with $C^\infty$ boundary. Let $A^{(j)},\tilde A^{(j)}: \overline{\Omega}\times \C\to S^j$, $j=1,2,3$, and $q, \tilde q:\overline{\Omega}\times\C\to \C$ satisfy the conditions (i), (ii), (iii). Assume furthermore that 
\begin{equation}
\label{eq_int_boundary_determination}
\p_z^kA^{(j)}(\cdot,0)|_{\p \Omega}= \p_z^k\tilde A^{(j)}(\cdot,0)|_{\p \Omega}, \quad  k=1,2,\dots, \quad j=1,2,3. 
\end{equation}
If $\Lambda_{A^{(1)},A^{(2)},A^{(3)},q}=\Lambda_{\tilde A^{(1)},\tilde A^{(2)}, \tilde A^{ (3)}, \tilde q}$ then $A^{ (j)}=\tilde A^{ (j)}$ in $\overline{\Omega}\times\C$, $j=1,2,3$, and $q=\tilde q$ in $\overline{\Omega}\times\C$. 
\end{thm}

\begin{rem}
To the best of our knowledge, the inverse boundary problem of recovering linear tensorial third-order perturbations of the biharmonic operator remains an open question. Furthermore, the uniqueness of such recovery is not possible due to the gauge invariance of the problem, as explained in \cite{Bhattacharyya_Ghosh_2021, Sahoo_Salo_2023}.  
We refer to \cite{Sahoo_Salo_2023} for the study of the linearized inverse problems for tensorial third-order perturbations of the biharmonic operator, where third-order tensorial perturbations are determined up to a natural gauge transformation. Theorem \ref{thm_main} can be regarded as the first result establishing the unique recovery of third-order nonlinear tensorial perturbations of biharmonic operators. 
\end{rem}

\begin{rem}
The assumption \eqref{eq_int_boundary_determination} in Theorem \ref{thm_main} can be removed by performing boundary determination, for instance as in \cite[Appendix C]{Krup_Uhlmann_magnetic}.
\end{rem}

Perturbed biharmonic operators naturally arise in various areas of physics and geometry, such as the study of the Kirchhoff plate equation in elasticity theory and the investigation of the Paneitz-Branson operator in conformal geometry, see \cite{Gazzola_Grunau_Sweers_book_2010, Selvadurai_book_2000}.  The inverse boundary problems concerning the recovery of zeroth-order linear perturbations of biharmonic operators have been studied in  \cite{Ikehata_1991, Isakov_1991},  see also \cite{Krupchyk_Uhlmanna_2016}. Furthermore, the investigations of inverse boundary problems for determining first-order linear perturbations of biharmonic operators, as well as more general polyharmonic operators, can be found in  \cite{Krup_Lassas_Uhlmann_2012, Krup_Lassas_Uhlmann_2014}. 

Further advancements have been made in the field of inverse boundary problems for linear perturbations of biharmonic and polyharmonic operators in \cite{Assylbekov_2016, Assylbekov_2017, Assylbekov_Iyer_2019,  Assylbekov_Yang_2017, Bhattacharyya_Ghosh_2019, Bhattacharyya_Ghosh_2021, Bhattacharyya_Krishnan_Sahoo_2021, Brown_Gauthier_2022, Choudhury_Krishnan_2015, Ghosh_Krishnan_2016, Liu_2020, Liu_2023,  Yan_2021, Yan_2023, Bhattacharyya_Kumar_2023, Agrawal_Jaiswal_Sahoo_2023, Aroua_Bellassoued_2022, Aroua_Bellassoued_2023}, see also  \cite{Bansal_Krishnan_Pattar_2023} for the two-dimensional case.  However, it should be noted that when considering linear perturbations of biharmonic operators specifically, all of these results are limited to determining up to second-order perturbations. The reason for this limitation lies in the construction of complex geometric optics solutions for the perturbed biharmonic operators, which are essential for solving the corresponding inverse problems. This construction relies on Carleman estimates with limiting Carleman weights. To obtain Carleman estimates for the perturbed semiclassical biharmonic operator, an iterative procedure is employed, involving the iteration of Carleman estimates with limiting Carleman weights for the semiclassical Laplacian and perturbing them with lower-order terms. However, it should be noted that the Carleman estimates with limiting Carleman weights for the semiclassical Laplacian suffer from a semiclassical loss, which is manifested by the presence of the factor '$h$' in the estimates, as illustrated in equation \eqref{eq_3_2} below. Consequently, the resulting Carleman estimates for the semiclassical biharmonic operator, as depicted in equation \eqref{eq_3_3} below, lack the necessary strength to accommodate third-order linear perturbations of the biharmonic operator.

Now the work \cite{Kurylev_Lassas_Uhlmann_2018} shows that nonlinearity might help when solving inverse problems for hyperbolic PDE. Furthermore, similar phenomena have been observed and utilized in the context of nonlinear elliptic equations, as demonstrated in \cite{Feizmohammadi_Oksanen_2020, Lassas_Liimatainen_Lin_Salo_2021}.  We refer to \cite{Kian_Krup_Uhlmann_2023, Krup_Uhlmann_2020,  Krup_Uhlmann_2020_remark, Krup_Uhlmann_magnetic, Krup_Uhlmann_Yan_magnetic, Lai_Ting, Mun_Uhl_2018, Lassas_Liimatainen_Lin_Salo_2021_partial, Liimatainen_Lin_Salo_Tyni_2022, Liimatainen_Lin_preprint_2022, Salo_Tzou_2023} for extensive recent work on inverse boundary problems for nonlinear elliptic PDE.  Notably, these works highlight that the presence of nonlinearity allows for solving inverse problems for nonlinear PDE, even when the corresponding inverse problems for linear equations remain unsolved. Our present paper aims to further illustrate and explore this intriguing phenomenon.

Let us now proceed to discuss the main ideas behind the proof of Theorem \ref{thm_main}. By employing the technique of higher-order linearizations of the Dirichlet--to--Neumann map, initially introduced in \cite{Kurylev_Lassas_Uhlmann_2018} (see also \cite{Feizmohammadi_Oksanen_2020, Lassas_Liimatainen_Lin_Salo_2021}), along with previous works on the second linearization \cite{Sun_1996, Sun_Uhlmann_1997}, the proof of Theorem \ref{thm_main} can be reduced to establishing the following density result. 
\begin{thm}
\label{thm_density}
Let $\Omega\subset \R^n$, $n\ge 3$, be a bounded open set with $C^\infty$ boundary. Let $A^{(j)}\in C(\overline{\Omega};  S^j)$ be such that 
\begin{equation}
\label{eq_int_boundary_determination_2}
A^{(j)}|_{\p \Omega}=0, 
\end{equation}
 $j=1,2,3$, and  $q\in L^\infty(\Omega)$. If  
\begin{equation}
\label{int_eq_6}
\begin{aligned}
\sum_{j=1}^3 \sum_{i_1,\dots, i_j=1}^n \int_{\Omega} A^{(j)}_{i_1\dots i_j}(x)\big( v^{(1)} D^j_{i_1\dots i_j}v^{(2)}+ v^{(2)} D^j_{i_1\dots i_j}v^{(1)}\big) v^{(0)}dx \\
+ \int_{\Omega} q(x)v^{(1)}v^{(2)}v^{(0)}dx=0,
\end{aligned}
\end{equation}
for all $v^{(l)}\in C^{4,\alpha}(\overline{\Omega})$ solving $(-\Delta)^2v^{(l)}=0$ in $\Omega$, $l=0,1,2$, 
then  $A^{(j)}=0$ and $q=0$ in $\Omega$, $j=1,2,3$.
\end{thm}
\begin{rem}
The assumption \eqref{eq_int_boundary_determination_2} in Theorem \ref{thm_density} can be removed by performing boundary determination, for instance as in \cite[Appendix C]{Krup_Uhlmann_magnetic}. This assumption ensures that when we extend $A^{(j)}$ to $\R^n\setminus \overline{\Omega}$ by setting it to zero outside of $\overline{\Omega}$ and denote this extension by the same letters, we satisfy $A^{(j)}\in C_0(\R^n; S^j)$ for $j=1,2,3$. The injectivity results for generalized momentum ray transforms, as presented in Lemma \ref{lem_inversion_1} and Lemma \ref{lem_inversion_2} below, apply precisely to this class of tensor fields. Our proof of Theorem \ref{thm_density} relies on these injectivity results. 
\end{rem}

To establish Theorem \ref{thm_density}, our first step involves constructing complex geometrical optics (CGO) solutions for the biharmonic equations, which belong to the function space $C^{4,\alpha}(\overline{\Omega})$. Since the integral identity \eqref{int_eq_6} involves products of three biharmonic functions and their derivatives up to the third order, it is advantageous to construct CGO solutions whose remainders with all derivatives up to the third order are small in $L^\infty$-norm. Additionally, we adopt specific amplitude choices as presented in \cite[Theorem 1.1]{Bhattacharyya_Krishnan_Sahoo_2021} for our constructed solutions.

By testing the integral identity  \eqref{int_eq_6} with the constructed solutions, we can effectively recover certain generalized momentum ray transforms on tensor fields as well as the sum of tensor fields of different ranks. The concept of momentum ray transforms was initially introduced in \cite{Sharafutdinov_1986}, with further investigations conducted in \cite{Sharafutdinov_bool_1994, Krishnan_Manna_Sahoo_Sharafutdinov_2019, Krishnan_Manna_Sahoo_Sharafutdinov_2020, Bhattacharyya_Krishnan_Sahoo_2021, Sahoo_Salo_2023, Ilmavirta_Kow_Sahoo_2023}. These transforms have proven instrumental in the study of inverse boundary problems for polyharmonic operators, as initially demonstrated in \cite{Bhattacharyya_Krishnan_Sahoo_2021} and also explored in \cite{Sahoo_Salo_2023,  Bhattacharyya_Kumar_2023}.
To conclude the proof of Theorem \ref{thm_density}, we leverage the injectivity results given in Lemma \ref{lem_inversion_1} and Lemma \ref{lem_inversion_2} below for the obtained momentum ray transforms.

\begin{rem}
Upon setting $v^{(2)}=1$ in \eqref{int_eq_6}, we obtain that
\begin{equation}
\label{int_eq_6_with_1}
\sum_{j=1}^3 \sum_{i_1,\dots, i_j=1}^n \int_{\Omega} A^{(j)}_{i_1\dots i_j}(x) (D^j_{i_1\dots i_j}v^{(1)}) v^{(0)}dx + \int_{\Omega} q(x)v^{(1)}v^{(0)}dx=0,
\end{equation}
for all $v^{(l)}\in C^{4,\alpha}(\overline{\Omega})$ solving $(-\Delta)^2v^{(l)}=0$ in $\Omega$, $l=0,1$. The integral identity \eqref{int_eq_6_with_1} corresponds to the same identity considered in \cite[Theorem 2.3]{Sahoo_Salo_2023}. In the aforementioned reference, the validity of \eqref{int_eq_6_with_1} is required for biharmonic functions $v^{(l)}\in H^4(\Omega)$, where $l=1,2$, and it is shown in \cite[Theorem 2.3]{Sahoo_Salo_2023} that the symmetric tensor of order three, denoted as $A^{(3)}$, can be recovered from this integral identity, up to a natural gauge. Therefore, it becomes apparent that the unique determination of a symmetric tensor of order three $A^{(3)}$ from \eqref{int_eq_6} necessitates the utilization of the freedom provided by employing three biharmonic functions in \eqref{int_eq_6}.
\end{rem}

\begin{rem}
In the case where $A^{(3)}=0$, it appears that to uniquely recover symmetric tensors $A^{(2)}$, $A^{(1)}$, and a function $q$ from \eqref{int_eq_6}, it is sufficient to work with two biharmonic functions. Indeed, assuming $A^{(3)}=0$ and setting $v^{(2)}=1$ in \eqref{int_eq_6}, we obtain the following identity,
\begin{equation}
\label{int_eq_6_with_1_second}
\sum_{j=1}^2 \sum_{i_1, i_j=1}^n \int_{\Omega} A^{(j)}_{i_1\dots i_j}(x) (D^j_{i_1 i_j}v^{(1)}) v^{(0)}dx + \int_{\Omega} q(x)v^{(1)}v^{(0)}dx=0,
\end{equation}
valid for all $v^{(l)}\in C^{4,\alpha}(\overline{\Omega})$ solving $(-\Delta)^2v^{(l)}=0$ in $\Omega$, $l=0,1$.  
The identity \eqref{int_eq_6_with_1_second} coincides with the identity considered in \cite[Theorem 2.1]{Sahoo_Salo_2023}, which holds for biharmonic functions $v^{(l)}\in H^4(\Omega)$. Consequently, following the proof of \cite[Theorem 2.1]{Sahoo_Salo_2023}, we can conclude that $A^{(j)}=0$ for $j=1,2$, and $q=0$ in $\Omega$.
\end{rem}

\begin{rem}
It is important to note that when establishing Theorem \ref{thm_density}, the recovery of $A^{(3)}$ is not independent of that of $A^{(2)}$. Initially, we can demonstrate that $A^{(3)}=i_\delta W^{(1)}$, where $W^{(1)}\in C_0(\R^n; S^{1})$. However, proving that $W^{(1)}=0$ requires the inversion of the generalized momentum ray transform, which involves both $W^{(1)}$ and the trace-free part of the tensor field $A^{(2)}$, as shown in \eqref{eq_4_22} below.
\end{rem}

Finally, it is worth mentioning that one could also explore the inverse boundary problem for the operator $L_{A^{(1)}, A^{(2)}, A^{(3)},q}$ with Navier boundary conditions instead of the Dirichlet conditions addressed in this paper. However, in the case of measurements taken on the entire boundary, the explicit description of the Laplacian in boundary normal coordinates, as outlined in \cite{Lee_Uhlmann_1989}, reveals that considering Navier boundary conditions will yield the same boundary data for the inverse problem.

The paper is organized as follows. In Section \ref{sec_tensors} we collect some facts on symmetric tensor fields and their momentum and generalized momentum ray transforms, including the injectivity results, which are required for proving Theorem \ref{thm_density}. For the completeness and convenience of the reader, we provide the proof of the necessary injectivity results for the generalized momentum ray transforms in Appendix \ref{sec_app_transforms}. The construction of complex geometric optics solutions to the biharmonic equations and the selection of appropriate amplitudes required for the proof of Theorem \ref{thm_density} are discussed in Section \ref{sec_CGO_solutions}. The proof of Theorem \ref{thm_density} is provided in Section \ref{sec_density}. Section \ref{sec_thm_main} is dedicated to the proof of Theorem \ref{thm_main}. Lastly, in Appendix \ref{sec_well-posedness}, the well-posedness of the Dirichlet problem for third-order nonlinear perturbations of biharmonic operators is demonstrated in the case of small boundary data.

\section{Momentum ray transforms}

\label{sec_tensors}

In this section we present some facts on symmetric tensor fields and their momentum and generalized momentum ray transforms, following \cite{Sharafutdinov_bool_1994, Bhattacharyya_Krishnan_Sahoo_2021, Sahoo_Salo_2023}, which are required for proving Theorem \ref{thm_density}.

\subsection{Symmetric tensor fields}

Let $n\ge 2$, and let $m\in \N\cup\{0\}$. Let  $T^m:=T^m(\R^n)$  be the complex vector space of functions $\underbrace{\R^n\times\dots\times\R^n}_{\text{m times}}\to \C$ which are $\R$--linear in each of its argument, and let $S^m:=S^m(\R^n)$ be its subspace which consists of functions symmetric in all arguments. The elements of $T^m$ are called tensors of degree $m$ while the elements of $S^m$ are said to be symmetric tensors of degree $m$. In particular, $S^0=\C$ and $S^1=\C^n$.  Let $\sigma: T^m\to S^m$ be the canonical projection (symmetrization) defined by 
\begin{equation}
\label{eq_5_1}
\sigma u(x_1,\dots, x_m)=\frac{1}{m!}\sum_{\pi\in \Pi_m} u (x_{\pi(1)}, \dots, x_{\pi(m)}),
\end{equation}
where the summation is taken over the group $\Pi_m$ of all permutations of the set $\{1,\dots, m\}$. 

Let $e_1,\dots, e_n$ be a basis of $\R^n$. Then the numbers $u_{i_1 \dots i_m}=u(e_{i_1}, \dots, e_{i_m})$ are said to be the covariant coordinates or components of $u\in T^m$. A tensor $u\in T^m$ belongs to $S^m$ if and only of its covariant components $u_{i_1 \dots i_m}$ are symmetric with respect to all indices. Together with  covariant coordinates $u_{i_1 \dots i_m}$ of the tensor $u$, one can define the contravariant coordinates $u^{i_1\dots i_m}$, see \cite[page 23]{Sharafutdinov_bool_1994}. In the case of the Euclidean metric tensor $g_{ij}=\delta_{ij}$, we have $u^{i_1\dots i_m}=u_{i_1\dots i_m}$.

The scalar product on $T^m$ is defined by 
\begin{equation}
\label{eq_5_2_scalar}
\langle u,v \rangle=\sum_{i_1,\dots, i_m=1}^n u^{i_1\dots i_m}\bar{v}_{i_1\dots i_m}, \quad u,v\in T^m. 
\end{equation}
For $ u\in T^k$ and $v\in T^m$, the tensor product $u\otimes v\in T^{k+m}$ is given by 
\[
u\otimes v(x_1,\dots, x_{k+m})=u(x_1,\dots, x_k)v(x_{k+1}, \dots, x_{k+m}).
\]

Together with the full symmetrization \eqref{eq_5_1}, given indices $i_1,\dots, i_p$, the operator of partial  symmetrization of a tensor $u\in T^m$ with respect to $i_1,\dots, i_p$, $p\le m$, is defined by 
\begin{equation}
\label{eq_5_2}
\sigma(i_1,\dots, i_p)u_{i_1,\dots, i_m}=\frac{1}{p!}\sum_{\pi\in \Pi_p} u_{i_{\pi(1)} \dots i_{\pi(p)} i_{p+1}\dots i_m}.
\end{equation}
 
We define the operator $i_\delta$ of symmetric multiplication by the Euclidean metric tensor $\delta$ as follows,
\begin{equation}
\label{eq_5_3}
i_\delta: S^m\to S^{m+2}, \quad (i_\delta u)_{i_1\dots i_{m+2}}=\sigma(i_1,\dots, i_{m+2})(u_{i_1\dots i_m}\delta_{i_{m+1}i_{m+2}}).
\end{equation}
We also define the operator $j_\delta$, dual to $i_\delta$, as follows, 
\begin{equation}
\label{eq_5_3_j}
j_{\delta}:S^m\to S^{m-2}, \quad (j_{\delta}u)_{i_1\dots i_{m-2}}=\sum_{i_{m-1}, i_m=1}^n u_{i_1\dots i_m}\delta^{i_{m-1}i_{m}}=\sum_{k=1}^n u_{i_1\dots i_{m-2}kk}.
\end{equation}
We also set $j_\delta f=0$ for $f\in S^0$ and $f\in S^1$. If for $f\in S^m$, $j_\delta f=0$, we say that $f$ is a trace free tensor. 

For $m\ge 0$, the following commutation formula holds on $S^m$, 
\begin{equation}
\label{eq_5_3_1}
j_\delta i_\delta=\frac{2(n+2m)}{(m+1)(m+2)}E+\frac{m(m-1)}{(m+1)(m+2)}i_\delta j_\delta,
\end{equation}
where $E$ is the identity operator on $S^m$, see \cite[Lemma 2.2]{Dairbekov_Sharafutdinov_2011}.

Let $\Omega\subset \R^n$ be open. We denote by $C^\infty_0(\Omega; S^m)$ the space of smooth compactly supported symmetric $m$--tensor fields on $\Omega$, 
and  by $\mathcal{E}'(\Omega; S^m)$ the space of compactly supported $m$--tensor field-distributions. 
We let $L^p(\Omega; S^m)$, $1\le p<\infty$, be  the space of $L^p$ integrable $m$--tensor fields on $\Omega$, and let $L^\infty(\Omega; S^m)$ be the spaces of bounded almost everywhere $m$--tensor fields on $\Omega$. We define the scalar product on $L^2(\Omega; S^m)$ by 
 \[
( u,v)_{L^2(\Omega;  S^m)}  =\int_{\R^n} \langle u(x),v(x)\rangle dx, 
\] 
where $\langle\cdot, \cdot\rangle $ is the scalar product on $S^m$ given by \eqref{eq_5_2_scalar}.

\subsection{Momentum ray transforms and their injectivity} 
Let 
\[
T\mathbb{S}^{n-1}=\{(x,\xi)\in \R^n\times \R^n: |\xi|=1,\  x\cdot\xi=0\}
\]
be the tangent bundle of the unit sphere $\mathbb{S}^{n-1}$, and let  $k\in \N\cup \{0\}$.  Following \cite{Sharafutdinov_1986}, see also
\cite{Sharafutdinov_bool_1994, Krishnan_Manna_Sahoo_Sharafutdinov_2019,  Krishnan_Manna_Sahoo_Sharafutdinov_2020, Bhattacharyya_Krishnan_Sahoo_2021, Sahoo_Salo_2023}, for $f^{(m)}\in C_0(\R^n; S^m)$ 
we define   the $k$th momentum ray transform $ \tilde I^{m,k}$ by 
\begin{equation}
\label{eq_6_1_tangent_bundle}
\tilde I^{m,k}f^{(m)}(x,\xi)=\sum_{i_1,\dots, i_m=1}^n\int_{\R} t^k f^{(m)}_{i_1\dots i_m}(x+t\xi)\xi_{i_1}\dots\xi_{i_m}dt, 
\end{equation}
where $(x,\xi)\in T\mathbb{S}^{n-1}$.  One can easily see that \begin{equation}
\label{eq_6_1_tangent_bundle_mapping}
 \tilde I^{m,k}: C_0(\R^n; S^m)\to C_0(T\mathbb{S}^{n-1}), \quad  \tilde I^{m,k}: C^\infty_0(\R^n; S^m)\to C^\infty_0(T\mathbb{S}^{n-1}),
\end{equation}
where $C_0(T\mathbb{S}^{n-1})$ ($C^\infty_0(T\mathbb{S}^{n-1})$) is the spaces of continuous (smooth) functions on $T\mathbb{S}^{n-1}$ that have compact support in the first variable. 

When $k=0$, $\tilde I^{m,0}$ is the classical ray transform. When $m=0$, the transform $\tilde I^{0,0}$ is the main mathematical tool of Computer Tomography. When $m=1$, the operator $\tilde I^{1,0}$ is known as the Doppler transform and is the key mathematical tool of Doppler Tomography. When $k>0$ and $m=0$, the operator $\tilde I^{0,k}$ appears in the study of the inversion of the cone transform and conical Radon transform whereas the latter transform has an application in Compton cameras, see \cite{Kuchment_Terzioglu_2017}.

We observe that \eqref{eq_6_1_tangent_bundle} also makes sense when $(x,\xi)\in \R^n\times (\R^n\setminus\{0\})$ and $f^{(m)}\in C_0(\R^n; S^m)$. We thus set  
\begin{equation}
\label{eq_6_1}
  I^{m,k}f^{(m)}(x,\xi)=\sum_{i_1,\dots, i_m=1}^n\int_{\R} t^k f^{(m)}_{i_1\dots i_m}(x+t\xi)\xi_{i_1}\dots\xi_{i_m}dt, 
\end{equation}
where $(x,\xi)\in \R^n\times (\R^n\setminus\{0\})$. A direct computation shows that the function $I^{m,k}f^{(m)}(x,\xi)$ has the following homogeneity property in the second variable, 
\begin{equation}
\label{eq_6_transfoem_2_var}
(I^{m,k}f^{(m)})(x,\tau\xi)= \frac{\tau^{m-k}}{|\tau|}(I^{m,k}f^{(m)})(x,\xi), \quad \tau\in \R\setminus\{0\},
\end{equation}
and the following property in the first variable,
\begin{equation}
\label{eq_6_transfoem_1_var}
(I^{m,k}f^{(m)})(x+\tau\xi,\xi)= \sum_{l=0}^k \begin{pmatrix} k\\
 l
 \end{pmatrix} (-\tau)^{k-l}
(I^{m,l}f^{(m)})(x,\xi), \quad \tau\in \R,
 \end{equation}
see \cite[formulas (2.3), (2.4)]{Krishnan_Manna_Sahoo_Sharafutdinov_2019}. It follows from \eqref{eq_6_transfoem_2_var} and \eqref{eq_6_transfoem_1_var} that
\begin{equation}
\label{eq_6_transfoem_both_var}
\begin{aligned}
(I^{m,k}f^{(m)})(x,\xi)= &|\xi|^{m-2k-1}\sum_{l=0}^k (-1)^{k-l}
 \begin{pmatrix} k\\
 l
 \end{pmatrix}
 |\xi|^{l} (\xi\cdot  x)^{k-l}\\
 &\times (I^{m,l}f^{(m)})\bigg(x-\frac{ \xi\cdot  x}{|\xi|^2}\xi,\frac{\xi}{|\xi|}\bigg), 
 \end{aligned}
\end{equation}
see \cite[formula (2.6)]{Krishnan_Manna_Sahoo_Sharafutdinov_2019}.  Note that here $(x-\frac{ \xi\cdot  x}{|\xi|^2}\xi, \frac{\xi}{|\xi|}) \in T\mathbb{S}^{n-1}$. Therefore, using \eqref{eq_6_1_tangent_bundle_mapping}, we see from \eqref{eq_6_transfoem_both_var} that 
\begin{equation}
\label{eq_6_1_mapping}
\begin{aligned}
&I^{m,k}: C_0(\R^n; S^m)\to C(\R^n\times(\R^n\setminus\{0\})), \\
 &I^{m,k}: C^\infty_0(\R^n; S^m)\to C^\infty(\R^n\times(\R^n\setminus\{0\})).
\end{aligned}
\end{equation}

We can also consider the $k$th momentum ray transform, defined by \eqref{eq_6_1_tangent_bundle}, for $(x,\xi)\in \R^n\times \mathbb{S}^{n-1}$ and $f^{(m)}\in C_0(\R^n; S^m)$. This transform will be denoted as $J^{m,k}$, i.e
\begin{equation}
\label{eq_6_1_J_new}
  J^{m,k}f^{(m)}(x,\xi)=\sum_{i_1,\dots, i_m=1}^n\int_{\R} t^k f^{(m)}_{i_1\dots i_m}(x+t\xi)\xi_{i_1}\dots\xi_{i_m}dt, 
\end{equation}
where $(x,\xi)\in \R^n\times  \mathbb{S}^{n-1}$.  We have $J^{m,k}f^{(m)}=I^{m,k}f^{(m)}|_{\R^n\times \mathbb{S}^{n-1}}$, and therefore, 
\begin{equation}
\label{eq_6_1_J_new_mapping}
J^{m,k}: C_0(\R^n; S^m)\to C(\R^n\times \mathbb{S}^{n-1}), \ J^{m,k}: C^\infty_0(\R^n; S^m)\to C^\infty(\R^n\times\mathbb{S}^{n-1} ).
\end{equation}
When $f^{(m)}\in C_0(\R^n; S^m)$,  \eqref{eq_6_transfoem_2_var} implies that 
\begin{equation}
\label{eq_201_0}
I^{m,k}f^{(m)}(x,\xi)=|\xi|^{m-k-1}J^{m,k}f^{(m)}\bigg(x,\frac{\xi}{|\xi|}\bigg), \quad x\in \R^n, \quad \xi\in \R^n\setminus\{0\}.
\end{equation}
Hence,  the operators $I^{m,k}$ and $J^{m,k}$, defined on $C_0(\R^n; S^m)$, provide equivalent information. 

In what follows, we shall only consider the momentum ray transforms $I^{m,k}$ and $J^{m,k}$. We shall next extend the definition \eqref{eq_6_1} of  $I^{m,k}$ to $\mathcal{E}'(\R^n; S^m)$. To that end, we let  $\mathcal{D}'(\R^n\times (\R^n\setminus\{0\}))$ be the spaces of distributions on $\R^n\times (\R^n\setminus\{0\})$. 
The $k$th momentum ray transform $I^{m,k}$ on $\mathcal{E}'(\R^n; S^m)$ is defined as follows, 
\begin{equation}
\label{eq_6_2_tras_def}
\begin{aligned}
I^{m,k}: \mathcal{E}'(\R^n; S^m)\to \mathcal{D}'(\R^n\times (\R^n\setminus\{0\})),\quad \langle I^{m,k} f^{(m)}, \psi\rangle=\langle f^{(m)}, (I^{m,k})^*\psi\rangle,
\end{aligned}
\end{equation}
where $\psi\in C^\infty_0(\R^n\times (\R^n\setminus\{0\}))$ and $(I^{m,k})^*$ is the adjoint of $I^{m,k}$ given by 
\begin{equation}
\label{eq_6_2}
\begin{aligned}
&(I^{m,k})^*: C^\infty_0(\R^n\times (\R^n\setminus\{0\}))\to C^\infty(\R^n; S^m),\\
&((I^{m,k})^*\psi)_{i_1\dots i_m}(x)=\int_{\R^n}\int_{\R}t^k\xi_{i_1}\dots\xi_{i_m}\psi(x-t\xi,\xi)dtd\xi, \quad 1\le i_1,\dots, i_m\le n,
\end{aligned}
\end{equation}
see \cite[Definition 4.2]{Bhattacharyya_Krishnan_Sahoo_2021}.  One can easily see that when $f^{(m)}\in C_0(\R^n;S^m)$, the $k$th momentum ray transform $I^{m,k}f^{(m)}$ defined by \eqref{eq_6_1} satisfies \eqref{eq_6_2_tras_def}.

Let $C_0(\R^n;\mathcal{S}^m):=C_0(\R^n;S^0)\oplus C_0(\R^n;S^1)\dots \oplus C_0(\R^n;S^m)$. Then any $F\in C_0(\R^n;\mathcal{S}^m)$ can be written uniquely as 
\begin{equation}
\label{eq_6_2_00}
F=\sum_{l=0}^m f^{(l)}=f^{(0)}+ \sum_{i_1=1}^n f^{(1)}_{i_1}dx^{i_1}+\dots + \sum_{i_1,\dots, i_m=1}^n f^{(m)}_{i_1\dots i_m}dx^{i_1}\dots dx^{i_m}.
\end{equation}
Let $k\in \N\cup \{0\}$. For $F\in C_0(\R^n; \mathcal{S}^m)$, the $k$th momentum ray transform $I^{m,k}$ is defined by 
\[
I^{m,k}F:=\sum_{l=0}^m I^{l,k} f^{(l)},
\]
where $I^{l,k}f^{(l)}$ is given by   \eqref{eq_6_1} with $m=l$.  Let $\mathcal{E}'(\R^n; \mathcal{S}^m):=\mathcal{E}'(\R^n;S^0)\oplus \mathcal{E}'(\R^n;S^1)\dots \oplus \mathcal{E}'(\R^n;S^m)$.  The $k$th momentum ray transform $I^{m,k}$ can be extended to $\mathcal{E}'(\R^n;\mathcal{S}^m)$ as follows,  
\begin{align*}
I^{m,k}: \mathcal{E}'(\R^n; \mathcal{S}^m)\to \mathcal{D}'(\R^n\times (\R^n\setminus\{0\})),\quad \langle I^{m,k} F, \psi\rangle=\sum_{l=0}^m\langle f^{(l)}, (I^{l,k})^*\psi\rangle,
\end{align*}
where $\psi\in C^\infty_0(\R^n\times (\R^n\setminus\{0\}))$ and $(I^{l,k})^*$ is given by \eqref{eq_6_2} with $m=l$, see  \cite[Definition 4.4]{Bhattacharyya_Krishnan_Sahoo_2021}.

Next, we state the injectivity results for the momentum ray transforms. Let us first recall that when $m\ge 1$, the ray transform $I^{m,0}$ always has a nontrivial kernel. Indeed, any symmetric tensor field can be decomposed uniquely  into its solenoidal and potential parts, see \cite[Theorem 2.6.2]{Sharafutdinov_bool_1994}, and the potential part lies in the kernel. 

\begin{thm}{\cite[Theorem 4.5]{Bhattacharyya_Krishnan_Sahoo_2021}}
Let $F\in \mathcal{E}'(\R^n; \mathcal{S}^m)$. If $I^{m,k}F=0$ for all $k=0,1,\dots, m$, then $F=0$. 
\end{thm}
\begin{cor}
Let $f^{(m)}\in \mathcal{E}'(\R^n; S^m)$. If $I^{m,k}f^{(m)}=0$ for all $k=0,1,\dots, m$, then $f^{(m)}=0$. 
\end{cor}

The following stronger injectivity result holds. 
\begin{thm}{\cite[Theorem 4.9]{Bhattacharyya_Krishnan_Sahoo_2021}}
\label{thm_higher_order_on}
Let $F\in \mathcal{E}'(\R^n; \mathcal{S}^m)$. If $I^{m,m}F=0$  then $F=0$. 
\end{thm}

We extend the operator $J^{m,k}$ to $\mathcal{E}'(\R^n; S^m)$ as follows, 
\begin{equation}
\label{eq_6_3_def_transf}
\begin{aligned}
J^{m,k}: \mathcal{E}'(\R^n;  S^m)\to \mathcal{D}'(\R^n\times \mathbb{S}^{n-1}),\quad \langle J^{m,k} f^{(m)}, \psi\rangle=\langle f^{(m)}, (J^{m,k})^*\psi\rangle.
\end{aligned}
\end{equation}
Here $\psi\in C^\infty_0(\R^n\times \mathbb{S}^{n-1})$ and $(J^{m,k})^*$ is the adjoint of $J^{m,k}$ given by 
\begin{equation}
\label{eq_6_3}
\begin{aligned}
&(J^{m,k})^*: C^\infty_0(\R^n\times \mathbb{S}^{n-1})\to C^\infty(\R^n; S^m),\\
&((J^{m,k})^*\psi)_{i_1\dots i_m}(x)=\int_{\mathbb{S}^{n-1}}\int_{\R}t^k\xi_{i_1}\dots\xi_{i_m}\psi(x-t\xi,\xi)dtdS(\xi), \  1\le i_1,\dots, i_m\le n,
\end{aligned}
\end{equation}
where $d S$ is the Euclidean surface element on the sphere $\mathbb{S}^{n-1}$, see \cite[Section 2.5]{Sharafutdinov_bool_1994}. One can easily see that when $f^{(m)}\in C_0(\R^n;S^m)$, the $k$th momentum ray transform  $J^{m,k} f^{(m)}$ defined by \eqref{eq_6_1_J_new}  satisfies \eqref{eq_6_3_def_transf}.

Let $F\in C_0(\R^n;\mathcal{S}^m)$ be given by \eqref{eq_6_2_00}. The $k$th momentum ray transform $J^{m,k}$ is defined by 
\[
J^{m,k}F:=\sum_{l=0}^m J^{l,k} f^{(l)},
\]
where $ J^{l,k} f^{(l)}$ is given by \eqref{eq_6_1_J_new} with $m=l$.  Furthermore, $J^{m,k}$ can be extended to $\mathcal{E}'(\R^n; \mathcal{S}^m)$ as follows,  
\begin{align*}
J^{m,k}: \mathcal{E}'(\R^n; \mathcal{S}^m)\to\mathcal{D}'(\R^n\times \mathbb{S}^{n-1}),\quad \langle J^{m,k} F, \psi\rangle=\sum_{l=0}^m\langle f^{(l)}, (J^{l,k})^*\psi\rangle,
\end{align*}
where $\psi\in C^\infty_0(\R^n\times \mathbb{S}^{n-1})$  and $(J^{l,k})^*$ is given by \eqref{eq_6_3} with $m=l$. 

In what follows, we shall need some properties of the $k$th momentum ray transform, which are obtained by a direct computation. When $f^{(m)}\in C_0(\R^n; S^m)$, we have 
\begin{equation}
\label{eq_201_1}
J^{m,k}f^{(m)}(x,-\xi)=(-1)^{m-k}J^{m,k}f^{(m)}(x,\xi),
\end{equation}
for all $(x,\xi)\in \R^n\times \mathbb{S}^{n-1}$. When  $f^{(0)}\in C_0(\R^n; S^0)$ and  $f^{(1)} \in C_0(\R^n; S^1)$, we also have
\begin{equation}
\label{eq_201_2}
J^{0, k}f^{(0)}(x,\xi)=J^{2,k}(i_\delta f^{(0)})(x,\xi),  \quad
J^{1,k}f^{(1)}(x,\xi)=J^{3,k}(i_\delta f^{(1)})(x,\xi),
\end{equation}
for all $(x,\xi)\in \R^n\times \mathbb{S}^{n-1}$.  Note that in \eqref{eq_201_2} it is crucial that $\xi\in \mathbb{S}^{n-1}$. Let $\xi\in\mathbb{S}^{n-1}$ and $f^{(m)}\in \mathcal{E}'(\R^n; S^m)$. Then we have 
\begin{equation}
\label{eq_6_4}
(\xi\cdot\p_x) J^{m,k}f^{(m)}=-kJ^{m,k-1}f^{(m)} \quad \text{in}\quad \mathcal{D}'(\R^n\times \mathbb{S}^{n-1}),
\end{equation}
see \cite[Theorem 4.8]{Bhattacharyya_Krishnan_Sahoo_2021}.

We shall need the following result established in \cite[Theorem 4.18]{Bhattacharyya_Krishnan_Sahoo_2021}. We state it only in the special cases of $m$--tensor fields with $m=0,1,2,3$ needed here. While the statement given in \cite[Theorem 4.18]{Bhattacharyya_Krishnan_Sahoo_2021} is for smooth with compact support tensor fields, the proof works in the case of continuous with compact support tensor fields as well, and we shall present it in Appendix \ref{sec_app_transforms_1} in the special cases needed for completeness, and convenience of the reader. 
\begin{thm}{\cite[Theorem 4.18]{Bhattacharyya_Krishnan_Sahoo_2021}}
\label{thm_Bhattacharyya_Krishnan_Sahoo}
Let $m=0,1,2,3$ and let $F=\sum_{l=0}^m f^{(l)}\in C_0(\R^n; \mathcal{S}^m)$. We have 
$J^{m,m}F(x,\xi)=0$, for all $(x,\xi)\in \R^n\times \mathbb{S}^{n-1}$ if and only if 
\begin{itemize}
\item[(i)] $f^{(2)}=-i_{\delta} f^{(0)}$ and $f^{(3)}=-i_{\delta} f^{(1)}$, when $m=3$, 
\item[(ii)] $f^{(2)}=-i_{\delta} f^{(0)}$ and $f^{(1)}=0$,  when $m=2$, 
\item[(iii)]  $f^{(1)}=0$ and $f^{(0)}=0$,  when $m=1$,
\item[(iv)]  $f^{(0)}=0$, when $m=0$. 
\end{itemize}
\end{thm}

\begin{rem}
It follows from Theorem \ref{thm_higher_order_on} and Theorem \ref{thm_Bhattacharyya_Krishnan_Sahoo} that while the kernel of $I^{m,m}$ on $ C_0(\R^n; \mathcal{S}^m)$ is trivial, the kernel of $J^{m,m}$  on $ C_0(\R^n; \mathcal{S}^m)$  is nontrivial when $m\ge 2$. 
\end{rem}

To prove Theorem \ref{thm_density}, we need to utilize a result related to the injectivity of the generalized momentum ray transform, as defined in \eqref{eq_5_4} below. This result has been established in \cite[Lemma 3.7]{Bhattacharyya_Krishnan_Sahoo_2021}, see also \cite[Lemma 3.1]{Sahoo_Salo_2023}. We will only state it in the specific cases where $m=1, 2, 3$ -- these cases are essential for this paper. The proof of this result can be found in Appendix \ref{sec_app_transforms_2} for completeness and convenience. 
\begin{lem}{\cite[Lemma 3.7]{Bhattacharyya_Krishnan_Sahoo_2021}}
\label{lem_inversion_1}
Let $n\ge 3$,   $e_1=(1,0,\dots, 0)\in \R^n$, and let $f^{(m)}\in C_0(\R^{n-1};  S^m(\R^n))$ with some $m=1,2,3$.  
Assume that 
\begin{equation}
\label{eq_5_4}
\sum_{i_1,\dots, i_m=1}^n \int_{\R} t^m f^{(m)}_{i_1\dots i_m}( y_0'+t\eta')(e_1+i\eta)_{i_1}\dots (e_1+i\eta)_{i_m}dt=0,
\end{equation}
for  all $y_0'\in \R^{n-1}$  and all $\eta=(\eta_1,\eta')\in \R^n$  such that $e_1\cdot\eta =0$ and $|\eta|=1$. Then 
\[
f^{(m)}=\begin{cases} i_{\delta} v^{(m-2)}, & m=2,3,\\
0, & m=1,
\end{cases}
\] 
where 
\[
v^{(1)}=\frac{3}{n+2}j_\delta f^{(3)}\in C_0(\R^{n-1}; S^{1}(\R^n)), \quad v^{(0)}=\frac{1}{n}j_\delta f^{(2)}\in  C_0(\R^{n-1}; S^{0}(\R^n)).
\] 
\end{lem}

We can derive the following immediate consequence of Lemma \ref{lem_inversion_1}. 
\begin{cor}
\label{cor_inversion_1}
For $m=2$ or $3$, assuming that all the conditions of Lemma \ref{lem_inversion_1} are met, and if, additionally, $j_\delta f^{(m)}=0$, it follows that $f^{(m)}=0$.
\end{cor}

In the proof of Theorem \ref{thm_density}, we will also require the following result, which addresses the injectivity of the generalized momentum ray transforms applied to a sum of tensor fields of order one and two, see \cite[Lemma 6.7]{Sahoo_Salo_2023} for a similar result. The proof of this result will be presented in Appendix \ref{sec_app_transforms_3}
\begin{lem}
\label{lem_inversion_2}
Let $n\ge 3$, $e_1=(1,0,\dots, 0)\in \R^n$, $f^{(2)}\in C_0(\R^{n-1}, S^2(\R^n))$, and $f^{(1)}\in C_0(\R^{n-1}, S^1(\R^n))$. 
Assume that 
\begin{equation}
\label{eq_5_25}
\begin{aligned}
\sum_{i_1,i_2=1}^n \int_{\R} t^{2+k} &f^{(2)}_{i_1 i_2}(y_0'+t\eta')(e_1+i\eta)_{i_1} (e_1+i\eta)_{i_2}dt\\
&+\sum_{i_1=1}^n \int_{\R} t^{1+k} f^{(1)}_{i_1}(y_0'+t\eta')(e_1+i\eta)_{i_1} dt
=0,
\end{aligned}
\end{equation}
for  all $y_0'\in \R^{n-1}$, all  $\eta=(\eta_1,\eta')\in \R^n$  such that $e_1\cdot\eta =0$ and $|\eta|=1$, and $k=0,1$. Then 
\[
f^{(2)}= i_{\delta} v^{(0)}, \quad f^{(1)}=0,
\] 
where $v^{(0)}= \frac{1}{n}j_\delta f^{(2)} \in C_0(\R^{n-1}, S^{0}(\R^n))$.  
\end{lem}

We have the following immediate consequence of Lemma \ref{lem_inversion_2}. 
\begin{cor}
\label{cor_inversion_2}
Assuming that all the conditions of Lemma \ref{lem_inversion_2} are met, and additionally, if $j_\delta f^{(2)}=0$, then we have that both $f^{(2)}=0$ and $f^{(1)}=0$.
\end{cor}

\section{Construction of biharmonic functions}

\label{sec_CGO_solutions}

Let $\Omega\subset \R^n$, $n\ge 3$, be a bounded open set with $C^\infty$ boundary.  In this section, we aim to construct complex geometric optics (CGO) solutions to the biharmonic equation
\begin{equation}
\label{eq_3_1}
\Delta^2 u = 0 \quad \text{in} \quad \Omega,
\end{equation}
which are necessary for proving Theorem \ref{thm_density}. Although the construction follows a standard approach as outlined in \cite{Krup_Lassas_Uhlmann_2012, Krup_Lassas_Uhlmann_2014, Bhattacharyya_Krishnan_Sahoo_2021}, we present it here to address the need for solutions with enhanced regularity. Specifically, the constructed CGO solutions should belong to $C^{4,\alpha}(\overline{\Omega})$ rather than the commonly used $H^4(\Omega)$, see the integral identity \eqref{int_eq_6} as well as \eqref{eq_2_5} and \eqref{eq_2_13} below. Additionally, since the integral identity \eqref{int_eq_6} in Theorem \ref{thm_density} involves products of three biharmonic functions and their derivatives up to the third order, it is advantageous to estimate the norms of the derivatives of the remainders up to the third order in the $L^\infty$-norm. Therefore, we need to construct CGO solutions that exhibit small remainders measured in the $C^3(\overline{\Omega})$-norm instead of the commonly used $H^4_{\text{scl}}(\Omega)$-norm in existing constructions. Consequently, we proceed with the construction of CGO solutions to \eqref{eq_3_1} that possess $C^5(\overline{\Omega})$ regularity, accompanied by small remainders in $C^5(\overline{\Omega})$.

Our starting point is the following Carleman estimate for the semiclassical Laplacian, which is established in \cite{Kenig_Sjostrand_Uhlmann_2007}. 

\begin{prop}
Let $\alpha\in \R^n\setminus\{0\}$ and let $\varphi(x)=\alpha\cdot x$. Then given any $s\in \R$, we have for all $h>0$ small enough and all $u\in C^\infty_0(\Omega)$, 
\begin{equation}
\label{eq_3_2}
h\|u\|_{H^s_{\emph{\text{scl}}}(\R^n)}\le C\|e^{\frac{\varphi}{h}}(-h^2\Delta)e^{-\frac{\varphi}{h}} u\|_{H^s_{\emph{\textrm{scl}}}(\R^n)}, \quad C>0. 
\end{equation}
\end{prop}
Here $H^s(\R^n)$, $s\in \R$, is the standard Sobolev space, equipped with the semiclassical norm $\|u\|_{H^s_{\text{scl}}(\R^n)}=\|\langle hD\rangle^s u \|_{L^2(\R^n)}$, where $\langle \xi\rangle =(1+|\xi|^2)^{1/2}$.  

Iterating the Carleman estimate \eqref{eq_3_2} twice, we obtain the following Carleman estimate for the biharmonic operator, 
\begin{equation}
\label{eq_3_3}
h^2\|u\|_{H^s_{\text{scl}}(\R^n)}\le C\|e^{\frac{\varphi}{h}}(-h^2\Delta)^2e^{-\frac{\varphi}{h}} u\|_{H^s_{\textrm{scl}}(\R^n)}, \quad C>0, 
\end{equation}
for all $u\in C^\infty_0(\Omega)$ and $h>0$ small enough. 

Using a standard argument, we convert the Carleman estimate \eqref{eq_3_3} into the following solvability result, see 
\cite[Proposition 2.3]{Krup_Uhlmann_2014}. 

\begin{prop}
\label{prop_solvability}
Let $\alpha\in \R^n\setminus\{0\}$,  let $\varphi(x)=\alpha\cdot x$, and let $s\in \R$. If $h>0$ is small enough, then for any $v\in H^s(\Omega)$, there is a solution $u\in H^s(\Omega)$ of the equation 
\[
e^{\frac{\varphi}{h}}(-h^2\Delta)^2e^{-\frac{\varphi}{h}} u=v\quad \text{in}\quad \Omega,
\]
which satisfies the bound
\[
\|u\|_{H^s_{\emph{\text{scl}}}(\Omega)}\le \frac{C}{h^2}\|v\|_{H^s_{\emph{\text{scl}}}(\Omega)}.
\]
\end{prop}
Here 
\[
H^s(\Omega)=\{V|_{\Omega}: V\in H^s(\R^n)\}, \quad s\in \R,
\]
equipped with the norm
\[
\|v\|_{H^s_{\text{scl}}(\Omega)}=\inf_{V\in H^s(\R^n), v=V|_{\Omega}}\|V\|_{H^s_{\text{scl}}(\R^n)}. 
\]

Let $l\in \N$. We shall construct CGO solutions to the equation \eqref{eq_3_1} in the form
\begin{equation}
\label{eq_3_4}
u(x;h)=e^{\frac{x\cdot \zeta}{h}}(a(x; h)+r(x; h)),
\end{equation}
where $\zeta\in \C^n$ such that $\zeta\cdot\zeta=0$, $|\text{Re}\, \zeta|=|\text{Im}\, \zeta|\sim 1$, the amplitude $a\in C^\infty(\overline{\Omega})$, and the remainder satisfies $\|r\|_{H^l_{\text{scl}}(\Omega)}\to 0$ as $h\to 0$. 

Now $u$ given by \eqref{eq_3_4} is a solution to \eqref{eq_3_1} provided that 
\begin{equation}
\label{eq_3_5}
e^{-\frac{x\cdot \zeta}{h}} (-h^2\Delta)^2 e^{\frac{x\cdot \zeta}{h}} r= - (-h^2\Delta-2\zeta\cdot h\nabla)^2a \quad \text{in}\quad \Omega. 
\end{equation}
Letting 
\[
T=\zeta\cdot\nabla,
\]
\eqref{eq_3_5} can be rewritten as follows, 
\begin{equation}
\label{eq_3_6}
\begin{aligned}
e^{-\frac{x\cdot \text{Re}\, \zeta}{h}} (-h^2\Delta)^2 e^{\frac{x\cdot \zeta}{h}} r&= - e^{\frac{i x\cdot \text{Im}\, \zeta}{h}}(h^2\Delta+2h T)^2a\\
&= - e^{\frac{i x\cdot \text{Im}\, \zeta}{h}}(h^4\Delta^2+ 4h^3 T\Delta+ 4h^2T^2)a  \quad \text{in}\quad \Omega. 
\end{aligned}
\end{equation}
We shall look for the amplitude in the form,
\[
a(x;h)=a_0(x)+ha_1(x)+\dots+ h^ka_k(x),
\]
for some $k\in \N$, where $a_0\in C^\infty(\R^n)$ satisfies the first transport equation,
\begin{equation}
\label{eq_3_7}
T^2 a_0= 0\quad \text{in}\quad \R^n,
\end{equation}
and $a_j\in C^\infty(\overline{\Omega})$ solves the $j$th transport equation, 
\begin{equation}
\label{eq_3_8} 
T^2a_j=-T\Delta a_{j-1}-\frac{1}{4}\Delta^2 a_{j-2} \quad \text{in}\quad \Omega,
\end{equation}
$j=1,2,\dots, k$. Here we set $a_{-1}=0$.

In the proof of Theorem \ref{thm_density}  we shall need CGO solutions  \eqref{eq_3_1}  of the form \eqref{eq_3_4} with $\zeta=\mu(e_1+i\eta_2)$, where $0\ne \mu\in \R$. Here $e_1=(1,0,\dots, 0)\in \R^n$ and $\eta_2\in \R^n$ is such that $e_1\cdot\eta_2=0$ and $|\eta_2|=1$. In this case, the first transport equation \eqref{eq_3_7} becomes 
\begin{equation}
\label{eq_3_9}
((e_1+i\eta_2)\cdot \nabla)^2 a_0= 0 \quad \text{in}\quad \R^n. 
\end{equation}
We shall proceed to solve \eqref{eq_3_9}. To that end, we complete $e_1$ and $\eta_2$ to an orthonormal basis in $\R^n$ and denote it by $\{e_1,\eta_2,\eta_3, \dots, \eta_n\}$. We denote by $y=(y_1,y_2,\dots, y_n)$, $y_1=x_1$, the coordinates with respect to this basis. In these coordinates \eqref{eq_3_9} becomes
\begin{equation}
\label{eq_3_10}
(\p_{y_1}+i\p_{y_2})^2 a_0= 0.
\end{equation}
Letting $z=y_1+iy_2$ be a complex variable, $\p_{\bar z}=\frac{1}{2}(\p_{y_1}+i\p_{y_2})$, and making a change of coordinates $y\mapsto (z,y'')$, $y''=(y_3,\dots, y_n)$,  the equation \eqref{eq_3_10} reduces to 
\begin{equation}
\label{eq_3_11}
\p_{\bar z}^2 a_0= 0.
\end{equation}
Functions $a_0(z)$ that satisfy \eqref{eq_3_11} are called polyanalytic of order two, see 
\cite{Balk_Zuev_1970}. One can easily see that any solution of \eqref{eq_3_11} in $\C$ has the form,
\[
a_0(z)=\bar z f_1(z)+f_2(z),
\]
where $f_1$ and $f_2$ are holomorphic in $\C$, see \cite[Page 203]{Balk_Zuev_1970}. Thus, in particular, the function
\begin{equation}
\label{eq_3_11_1}
\big((2i)^{-1}(z-\bar z)-c\big)^k f(z) g(y''),
\end{equation}
satisfies \eqref{eq_3_11}. Here $k=0,1$, $c\in \R$,  $f$ is holomorphic in $\C$, and  $g\in C^\infty(\R^{n-2})$. In view of \eqref{eq_3_11_1}, in the proof of Theorem \ref{thm_density} we shall work with solutions $a_0\in C^\infty(\R^n)$ to \eqref{eq_3_10} in the form, 
\begin{equation}
\label{eq_3_12}
a_0(y_1,y_2,y'')=(y_2-c)^k f(y_1+iy_2) g(y'').
\end{equation}
The particular form \eqref{eq_3_12} of solutions $a_0$ of the first transport equation \eqref{eq_3_10} will be helpful to get the generalized momentum ray transforms of unknown perturbations of $\Delta^2$, see \cite{Bhattacharyya_Krishnan_Sahoo_2021, Sahoo_Salo_2023} for the same choice of the amplitude.

We shall now proceed to solve the $j$th transport equation \eqref{eq_3_8}, assuming that we have already found $a_0, a_1, \dots, a_{j-1} \in C^\infty(\overline{\Omega})$ for $j=1, \dots, k$.
Making the change  of coordinates $x\mapsto (z,y'')$,  \eqref{eq_3_8} becomes
\begin{equation}
\label{eq_3_13}
\p_{\bar z}^2 a_j=f \quad \text{in}\quad \Omega,
\end{equation}
where $f\in C^\infty(\overline{\Omega})$ is given.  As in \cite{Krup_Lassas_Uhlmann_2012}, in order to solve \eqref{eq_3_13} we first find $v\in C^\infty(\overline{\Omega})$ which satisfies 
\begin{equation}
\label{eq_3_14}
\p_{\bar z} v=f \quad \text{in}\quad \Omega,
\end{equation}
and then solve 
\begin{equation}
\label{eq_3_15}
\p_{\bar z} a_j=v \quad \text{in}\quad \Omega.
\end{equation}
Note that \eqref{eq_3_14} can be solved by applying the Cauchy transform, i.e. 
\[
v(z,y'')=\frac{1}{\pi}\int_{\C}\frac{\chi(z-\zeta,y'')f(z-\zeta,y'')}{\zeta}d\text{Re}\,\zeta d\text{Im}\,\zeta, 
\]
where $\chi\in C^\infty_0(\R^n)$ is such that $\chi=1$ near $\overline{\Omega}$. The equation \eqref{eq_3_15} can be solved in a similar way and therefore, \eqref{eq_3_13} and hence, \eqref{eq_3_8} are solvable globally near $\overline{\Omega}$ with a solution $a_j\in C^\infty(\overline{\Omega})$.  

Having chosen the amplitudes $a_0, a_1,\dots, a_k\in C^\infty(\overline{\Omega})$, we get from \eqref{eq_3_6} that 
\begin{equation}
\label{eq_3_16}
e^{-\frac{\mu x_1}{h}} (-h^2\Delta)^2 e^{\frac{\mu x_1}{h}} (e^{\frac{i \mu x\cdot \eta_2}{h}} r)= - e^{\frac{i \mu x\cdot \eta_2}{h}}h^{k+3} (h\Delta^2 a_k+ 4T\Delta a_k)  \quad \text{in}\quad \Omega. 
\end{equation}

Let  $l\in \N$. Then Proposition \ref{prop_solvability} implies that for all $h>0$ small enough,  there is $r\in H^l(\Omega)$ satisfying \eqref{eq_3_16} such that 
\[
\|r\|_{H^l_{\text{scl}}(\Omega)}= \mathcal{O}(h^{k+1}).
\]
Here we used that $\|r\|_{H^l_{\text{scl}}(\Omega)}^2\sim \sum_{|\alpha|\le l}\|(hD)^\alpha r\|^2_{L^2(\Omega)}$.  Summing up, we have the following result. 

\begin{prop}
\label{prop_CGO}
Let $\mu\in \R\setminus\{0\}$, $e_1=(1,0,\dots, 0)\in \R^n$,  and $\eta_2\in \R^n$ be such that $e_1\cdot\eta_2=0$ and $|\eta_2|=1$. Let $k, l\in \N$. Then for all $h>0$ small enough, there is $u\in H^l(\Omega)$ satisfying $\Delta^2 u=0$ in $\Omega$ of the form,
\[
u(x;h)=e^{\frac{\mu x\cdot (e_1+i\eta_2)}{h}}(a_0(x)+ha_1(x)+\dots+h^ka_k(x)+r(x; h)).
\]
Here $a_0\in C^\infty(\R^n)$ satisfies the first transport equation \eqref{eq_3_9}, and $a_j\in C^\infty(\overline{\Omega})$ satisfies the $j$th transport equation \eqref{eq_3_8},  $j=1,\dots, k$, and $r\in H^l(\Omega)$ such that $\|r\|_{H^l_{\text{scl}}(\Omega)}= \mathcal{O}(h^{k+1})$, as $h\to 0$. 
\end{prop}

Now 
\[
\|r\|_{H^l(\Omega)}\le h^{-l}\|r\|_{H^l_{\text{scl}}(\Omega)}=\mathcal{O}(h^{k+1-l}).
\]
Taking $l>n/2+5$ and using the Sobolev embedding $H^l(\Omega)\subset C^5(\overline{\Omega})$, see \cite[Theorem 5.1.4]{Agranovich_2015}, we see that $r\in C^5(\overline{\Omega})$ with 
\begin{equation}
\label{eq_3_17}
\|r\|_{C^5(\overline{\Omega})}=\mathcal{O}(h^{k+1-l}),
\end{equation}
as $h\to 0$. Here 
\[
\|r\|_{C^5(\overline{\Omega})}=\sum_{|\alpha|\le 5}\sup_{\overline{\Omega}} |D^\alpha r|.
\]
Letting $k=l$, in view of \eqref{eq_3_17}, we get the following consequence of Proposition \ref{prop_CGO}. 
\begin{cor}
\label{cor_CGO}
Let $\mu\in \R\setminus\{0\}$, $e_1=(1,0,\dots, 0)\in \R^n$,  and $\eta_2\in \R^n$ be such that $e_1\cdot\eta_2=0$ and $|\eta_2|=1$. Then for all $h>0$ small enough, there is $u\in C^5(\overline{\Omega})$ satisfying $\Delta^2 u=0$ in $\Omega$ of the form,
\begin{equation}
\label{eq_cor_CGO-1}
u(x;h)=e^{\frac{\mu x\cdot (e_1+i\eta_2)}{h}}(a_0(x)+r(x; h)).
\end{equation}
Here $a_0\in C^\infty(\R^n)$ satisfies the first transport equation \eqref{eq_3_9}, and $r\in C^5(\overline{\Omega})$ such that $\|r\|_{C^5(\overline{\Omega})}= \mathcal{O}(h)$, as $h\to 0$. 
\end{cor}

\begin{rem}
\label{rem_choice_amplitude}
In the proof of Theorem \ref{thm_density}, to obtain the generalized momentum ray transforms of unknown tensorial perturbations of $\Delta^2$, we shall work with $a_0$ in \eqref{eq_cor_CGO-1} given by 
\[
a_0(y_1,y_2,y'')=(y_2-c)^k f(z) g(y''),
\]
where $k=0,1$, $c\in\R$, $f$ is holomorphic in $\C$, and $g\in C^\infty(\R^{n-2})$,  see \eqref{eq_3_12}.  Here $y=(y_1,y_2,\dots, y_n)$, $y_1=x_1$,  are  the coordinates with respect to the orthonormal basis
$\{e_1,\eta_2,\eta_3, \dots, \eta_n\}$, and $z=y_1+iy_2$. We refer to 
\cite{Bhattacharyya_Krishnan_Sahoo_2021, Sahoo_Salo_2023} where a similar choice of the amplitude was made when solving inverse problems.  
\end{rem}

In the proof of Theorem \ref{thm_density}, to recover the Fourier transform of $q$, we will require the following CGO solutions to biharmonic equations, constructed in the same way as before.
\begin{prop}
\label{prop_CGO_q}
Let $\alpha, \beta\in \R^n$ be such that $\alpha\cdot\beta=0$ and $|\alpha|=|\beta|=1$. Then for all $h>0$ small enough, there is $u\in C^5(\overline{\Omega})$ satisfying $\Delta^2 u=0$ in $\Omega$ of the form,
\begin{equation}
\label{eq_cor_CGO-2}
u(x;h)=e^{\frac{ x\cdot (\alpha+i\beta)}{h}}(a_0(x)+r(x; h)),
\end{equation}
where  $a_0\in C^\infty(\R^n)$ satisfies $((\alpha+i\beta)\cdot \nabla)^2 a_0=0$ in $\R^n$. Here $\|r\|_{C^5(\overline{\Omega})}= \mathcal{O}(h)$, as $h\to 0$. 
\end{prop}

\section{Proof of Theorem \ref{thm_density}}

\label{sec_density}

We shall start to examine the integral identity \eqref{int_eq_6} by employing biharmonic functions constructed in Corollary \ref{cor_CGO}. To that end, following \cite[Theorem 1.1]{Bhattacharyya_Krishnan_Sahoo_2021}, we carefully select appropriate amplitudes as discussed in Remark \ref{rem_choice_amplitude}. This approach allows us to recover specific generalized momentum ray transforms on tensor fields and the sum of tensor fields of different ranks. Leveraging the injectivity results of Lemma \ref{lem_inversion_1} and Lemma \ref{lem_inversion_2} for the obtained generalized momentum ray transforms, we shall show that $A^{(j)}=0$ in $\Omega$, $j=1,2,3$. To demonstrate that $q=0$, we shall utilize the biharmonic functions from Proposition \ref{prop_CGO_q} to recover the Fourier transform of $q$. 

In doing so, we first extend $A^{(j)}$ by zero to $\R^n\setminus\overline{\Omega}$ and denote the extension again by $A^{(j)}$, $j=1,2,3$. We see that $A^{(j)}\in C_0(\R^n; S^j)$  as $A^{(j)}|_{\p \Omega}=0$, $j=1,2,3$. 
Let $e_1=(1,0,\dots, 0)\in \R^n$ and let $\eta_2\in \R^n$ be  such that $e_1\cdot\eta_2=0$ and $|\eta_2|=1$. Thus, $\eta_2=(0,\eta_2')$, $\eta_2'\in \R^{n-1}$, $|\eta_2'|=1$.  By Corollary \ref{cor_CGO}, there are $v^{(0)}, v^{(1)}, v^{(2)}\in C^{5}(\overline{\Omega})$ satisfying $(-\Delta)^2 v^{(l)}=0$ in $\Omega$ of the form 
\begin{equation}
\label{eq_4_1}
\begin{aligned}
v^{(0)}&=e^{\frac{-2x\cdot (e_1+i\eta_2)}{h}}(a_0^{(0)}(x)+r^{(0)}(x; h)),\\
v^{(1)}&=e^{\frac{x\cdot (e_1+i\eta_2)}{h}}(a_0^{(1)}(x)+r^{(1)}(x; h)),\quad v^{(2)}=v^{(1)}.
\end{aligned}
\end{equation}
Here $a_0^{(l)}\in C^\infty(\R^n)$ satisfies \eqref{eq_3_9},  and 
\begin{equation}
\label{eq_4_2}
\|r^{(l)}\|_{C^5(\overline{\Omega})}=\mathcal{O}(h),
\end{equation}
as $h\to 0$,  $l=0,1$.

Substituting $v^{(0)}, v^{(1)}, v^{(2)}$ given by \eqref{eq_4_1} in the integral identity \eqref{int_eq_6}, multiplying by $h^3$, letting $h\to 0$, and using \eqref{eq_4_2}, we obtain that 
\begin{equation}
\label{eq_4_3}
\sum_{i_1,i_2,i_3=1}^n \int_{\Omega} A^{(3)}_{i_1 i_2 i_3}(x) (e_1+i\eta_2)_{i_1}(e_1+i\eta_2)_{i_2}(e_1+i\eta_2)_{i_3} (a_0^{(1)}(x))^2 a_0^{(0)}(x)dx=0.
\end{equation}
Here $(e_1+i\eta_2)_{i_j}$ is the $i_j$th component of the vector $e_1+i\eta_2$ in the standard basis in $\R^n$. 
Now we have the following decomposition 
\begin{equation}
\label{eq_4_3_decom_new}
A^{(3)}=B^{(3)}+i_\delta W^{(1)},
\end{equation}
where $W^{(1)}\in C_0(\R^n; S^1)$, $B^{(3)}\in C_0(\R^n;S^3)$ is such that 
\begin{equation}
\label{eq_4_3_decom_new_2}
j_\delta B^{(3)}=0, 
\end{equation}
see \cite[Lemma 2.3]{Dairbekov_Sharafutdinov_2011} as well as  the beginning of the proof of Lemma \ref{lem_inversion_1} in Appendix \ref{sec_app_transforms_2}. Since  $(e_1+i\eta_2)\cdot(e_1+i\eta_2)=0$, we have 
\begin{equation}
\label{eq_4_16}
\sum_{i_1,i_2,i_3=1}^n (i_\delta W^{(1)})_{i_1 i_2 i_3} (e_1+i\eta_2)_{i_1}(e_1+i\eta_2)_{i_2}(e_1+i\eta_2)_{i_3}=0.
\end{equation}

Substituting \eqref{eq_4_3_decom_new} into \eqref{eq_4_3} and using \eqref{eq_4_16}, we get 
\begin{equation}
\label{eq_4_3_decom_new_3}
\sum_{i_1,i_2,i_3=1}^n \int_{\R^n} B^{(3)}_{i_1 i_2 i_3}(x) (e_1+i\eta_2)_{i_1}(e_1+i\eta_2)_{i_2}(e_1+i\eta_2)_{i_3} (a_0^{(1)}(x))^2 a_0^{(0)}(x)dx=0.
\end{equation}
We shall proceed to show that $B^{(3)}=0$. In doing so, we shall complete  $e_1$ and $\eta_2$ to an orthonormal basis $\{e_1,\eta_2,\eta_3, \dots, \eta_n\}$ in $\R^n$, and write $y=(x_1,y_2,y'')$, $y''=(y_3,\dots, y_n)$, for the coordinates with respect to this basis.  According to Remark \ref{rem_choice_amplitude}, we choose
\begin{equation}
\label{eq_4_5}
a_0^{(1)}(x_1,y_2,y'')=y_2-c, \quad a_0^{(0)}(x_1,t,y'')=(y_2-c)e^{-i\lambda(x_1+iy_2)}g(y''), 
\end{equation}
where $c\in\R$,  $\lambda\in \R$, and $g\in C^\infty(\R^{n-2})$. Writing out the integral in  \eqref{eq_4_3_decom_new_3}  in $y$--coordinates, and  substituting \eqref{eq_4_5}, we get 
\begin{equation}
\label{eq_4_6}
\begin{aligned}
\sum_{i_1,i_2,i_3=1}^n \int_{\R^n} B^{(3)}_{i_1 i_2 i_3}(x_1,y_2,y'') (e_1+i\eta_2)_{i_1}&(e_1+i\eta_2)_{i_2}(e_1+i\eta_2)_{i_3}\\
& (y_2-c)^3 e^{-i\lambda(x_1+iy_2)}g(y'') dx_1dy_2dy''=0.
\end{aligned}
\end{equation}
We get from \eqref{eq_4_6} that 
\begin{equation}
\label{eq_4_6_new_4_6}
\begin{aligned}
 \sum_{i_1,i_2,i_3=1}^n  \int_{\R^{n-1}} \hat {B^{(3)}_{i_1 i_2 i_3}}(\lambda,y_2,y'') (e_1+i\eta_2)_{i_1}&(e_1+i\eta_2)_{i_2}(e_1+i\eta_2)_{i_3}\\
& (y_2-c)^3 e^{\lambda y_2} g(y'') dy_2 dy''=0,
\end{aligned}
\end{equation}
for all $c\in\R$,  $\lambda\in \R$, and $g\in C^\infty(\R^{n-2})$. 
Here 
\[
\hat {B^{(3)}}(\lambda,y')=\int_{\R} e^{-i \lambda x_1}B^{(3)}(x_1,y')dx_1
\]
is the Fourier transform of $B^{(3)}$ with respect to $x_1$, and $y'=(y_2,\dots, y_n)\in \R^{n-1}$. We observe that the function $x_1\mapsto B^{(3)}(x_1,y')$ is of class  $C_0(\R; S^3(\R^n))$ for any $y'\in \R^{n-1}$. Furthermore, it can be easily verified that the function $y'\mapsto \hat {B^{(3)}}(\lambda,y')$ is of class $C_0(\R^{n-1}; S^3(\R^n))$ for any $\lambda\in \R$. Based on these observations, one can easily check that, for any $\lambda\in \R$ and $c\in \R$, the integral with respect to $y_2$ in \eqref{eq_4_6_new_4_6} is a continuous function of $y''\in\R^{n-2}$.  Now since $g\in C^\infty(\R^{n-2})$ is arbitrary, we conclude from \eqref{eq_4_6_new_4_6} that 
\begin{equation}
\label{eq_4_7}
\sum_{i_1,i_2,i_3=1}^n \int_{\R} \hat {B^{(3)}_{i_1 i_2 i_3}}(\lambda,y_2, y'') (e_1+i\eta_2)_{i_1}(e_1+i\eta_2)_{i_2}(e_1+i\eta_2)_{i_3} (y_2-c)^3 e^{\lambda y_2}dy_2=0,
\end{equation}
for all $y''\in \R^{n-1}$, all $c\in \R$, and all $\lambda\in \R$.  

Making the change of variables $t= y_2-c$ in the integral in \eqref{eq_4_7}, we obtain that
\begin{equation}
\label{eq_4_7_2}
\sum_{i_1,i_2,i_3=1}^n \int_{\R} t^3 \hat {B^{(3)}_{i_1 i_2 i_3}}(\lambda,y_0'+t\eta'_2) (e_1+i\eta_2)_{i_1}(e_1+i\eta_2)_{i_2}(e_1+i\eta_2)_{i_3}  e^{\lambda t}dt=0,
\end{equation}
for all $y_0'=(c,y'')\in \R^{n-1}$, and all $\lambda\in \R$.  Now, since the function $x_1\mapsto B^{(3)}(x_1,y')$ is of class $C_0(\R; S^{3}(\R^n))$ for any $y'\in \R^{n-1}$, we have that the function $\lambda\mapsto \hat {B^{(3)}}(\lambda,y')$ is of class $C^\infty(\R; S^3(\R^n))$ for any $y'\in \R^{n-1}$.  

Letting $\lambda=0$ in \eqref{eq_4_7_2}, we get 
\begin{equation}
\label{eq_4_8}
\sum_{i_1,i_2,i_3=1}^n \int_{\R}  t^3  \hat {B^{(3)}_{i_1 i_2 i_3}}(0,y_0'+t\eta_2') (e_1+i\eta_2)_{i_1}(e_1+i\eta_2)_{i_2}(e_1+i\eta_2)_{i_3} dt=0,
\end{equation}
for  all $y_0'\in \R^{n-1}$, and all $\eta_2\in \R^{n}$ such that $e_1\cdot \eta_2=0$ and $|\eta_2|=1$.
Here $\hat{B^{(3)}}(0, \cdot)\in C_0(\R^{n-1};S^3(\R^n))$ and $j_\delta  (\hat {B^{(3)}}(0,\cdot))=0$ in view of \eqref{eq_4_3_decom_new_2} and the fact that $j_\delta$ commutes with the Fourier transform.  An application of Corollary \ref{cor_inversion_1} to \eqref{eq_4_8} shows that 
\begin{equation}
\label{eq_4_9}
\hat {B^{(3)}}(0,\cdot)=0.
\end{equation} 

Differentiation  \eqref{eq_4_7_2} with respect to $\lambda$,  letting $\lambda=0$, and using \eqref{eq_4_9}, we obtain that 
\begin{equation}
\label{eq_4_10}
\begin{aligned}
\sum_{i_1,i_2,i_3=1}^n& \int_{\R} t^3 \p_{\lambda} \hat {B^{(3)}_{i_1 i_2 i_3}}(0,y_0'+t\eta') (e_1+i\eta_2)_{i_1}(e_1+i\eta_2)_{i_2}(e_1+i\eta_2)_{i_3}  dt=0,
\end{aligned}
\end{equation}
for  all $y_0'\in \R^{n-1}$, and all $\eta_2\in \R^{n}$ such that $e_1\cdot \eta_2=0$ and $|\eta_2|=1$.
Here $ \p_{\lambda} \hat {B^{(3)}}(0, \cdot)\in C_0(\R^{n-1};S^3(\R^n))$ and $j_\delta (\p_{\lambda} \hat {B^{(3)}}(0, \cdot))=0$ in view of \eqref{eq_4_3_decom_new_2}. 

Applying Corollary \ref{cor_inversion_1} to \eqref{eq_4_10}, we conclude that $\p_{\lambda}\hat {B^{(3)}}(0,\cdot)=0$. Continuing in the same fashion, we show that $\p_{\lambda}^k\hat {B^{(3)}}(0,\cdot)=0$ for all $k=0,1,2,\dots$.

Since $B^{(3)}(x_1,\cdot)$ is compactly supported in $x_1$, the function $\lambda\mapsto\hat {B^{(3)}}(\lambda,\cdot)$ is analytic. Therefore, we can deduce that $\hat {B^{(3)}}(\lambda,\cdot)=0$ for all $\lambda\in \R$. Consequently, $B^{(3)}=0$, and it follows from \eqref{eq_4_3_decom_new} that
\begin{equation}
\label{eq_4_15}
A^{(3)}=i_\delta W^{(1)}.
\end{equation}

Let us now return to the integral identity \eqref{int_eq_6} with $A^{(3)}$ given by \eqref{eq_4_15} and substitute $v^{(0)}, v^{(1)}, v^{(2)}$ given by \eqref{eq_4_1} in \eqref{int_eq_6}. By multiplying \eqref{int_eq_6} by $h^2$, letting $h\to 0$, and utilizing \eqref{eq_4_16}, we obtain that 
\begin{equation}
\label{eq_4_17}
\begin{aligned}
&\sum_{i_1,i_2=1}^n \int_{\Omega} A^{(2)}_{i_1 i_2}(x) (e_1+i\eta_2)_{i_1}(e_1+i\eta_2)_{i_2}(a_0^{(1)}(x))^2 a_0^{(0)}(x)dx +\frac{3}{i}\times\\
& \sum_{i_1,i_2,i_3=1}^n \int_{\Omega} (i_\delta W^{(1)})_{i_1 i_2 i_3}(x)(e_1+i\eta_2)_{i_1}(e_1+i\eta_2)_{i_2}(\p_{i_3} a_0^{(1)}(x))a_0^{(1)}(x) a_0^{(0)}(x)dx=0.
\end{aligned}
\end{equation}
Using \eqref{eq_5_3}, \eqref{eq_5_2}, and the fact that $(e_1+i\eta_2)\cdot(e_1+i\eta_2)=0$, we get from \eqref{eq_4_17} that 
\begin{equation}
\label{eq_4_18}
\begin{aligned}
&\sum_{i_1,i_2=1}^n \int_{\R^n} A^{(2)}_{i_1 i_2}(x) (e_1+i\eta_2)_{i_1}(e_1+i\eta_2)_{i_2}(a_0^{(1)}(x))^2 a_0^{(0)}(x)dx\\
&+
\frac{2}{i}\sum_{i_1=1}^n  \int_{\R^n}  W^{(1)}_{i_1}(x)(e_1+i\eta_2)_{i_1}( (e_1+i\eta_2)\cdot \nabla a_0^{(1)}(x))a_0^{(1)}(x) a_0^{(0)}(x)dx=0.
\end{aligned}
\end{equation}

Now we express $A^{(2)}$ as
\begin{equation}
\label{eq_4_3_decom_new_for_2_tensors}
A^{(2)}=B^{(2)}+i_\delta W^{(0)},
\end{equation}
where $W^{(0)}\in C_0(\R^n; S^0)$, $B^{(2)}\in C_0(\R^n;S^2)$ is such that 
\begin{equation}
\label{eq_4_3_decom_new_2_for_2_tensors}
j_\delta B^{(2)}=0, 
\end{equation}
see \cite[Lemma 2.3]{Dairbekov_Sharafutdinov_2011} as well as  the proof of Lemma \ref{lem_inversion_1} in Appendix \ref{sec_app_transforms_2} in the case $m=2$. Thanks to the fact that $(e_1+i\eta_2)\cdot (e_1+i\eta_2)=0$, we get
\begin{equation}
\label{eq_4_3_decom_new_2_for_2_tensors_2}
\sum_{i_1,i_2=1}^n (i_\delta W^{(0)})_{i_1 i_2} (e_1+i\eta_2)_{i_1}(e_1+i\eta_2)_{i_2}=0.
\end{equation}
By substituting \eqref{eq_4_3_decom_new_for_2_tensors} into \eqref{eq_4_18} and using \eqref{eq_4_3_decom_new_2_for_2_tensors_2}, we obtain 
\begin{equation}
\label{eq_4_18_new_decomp}
\begin{aligned}
&\sum_{i_1,i_2=1}^n \int_{\R^n} B^{(2)}_{i_1 i_2}(x) (e_1+i\eta_2)_{i_1}(e_1+i\eta_2)_{i_2}(a_0^{(1)}(x))^2 a_0^{(0)}(x)dx\\
&+
\frac{2}{i}\sum_{i_1=1}^n  \int_{\R^n}  W^{(1)}_{i_1}(x)(e_1+i\eta_2)_{i_1}( (e_1+i\eta_2)\cdot \nabla a_0^{(1)}(x))a_0^{(1)}(x) a_0^{(0)}(x)dx=0.
\end{aligned}
\end{equation}

As before, we shall work in the coordinates $y=(x_1,y_2,y'')$, $y''=(y_3,\dots, y_n)$, with respect to the orthonormal basis $\{e_1,\eta_2,\eta_3, \dots, \eta_n\}$ in $\R^n$.  According to Remark \ref{rem_choice_amplitude}, we choose  
\begin{equation}
\label{eq_4_19}
a_0^{(1)}(x_1,y_2,y'')=y_2-c, \quad a_0^{(0)}(x_1,t,y'')=(y_2-c)^l e^{-i\lambda(x_1+iy_2)}g(y''), 
\end{equation}
where $c\in\R$,  $\lambda\in \R$,  $g\in C^\infty(\R^{n-2})$, and $l=0,1$. 

Writing out the integral in  \eqref{eq_4_18_new_decomp}  in $y$--coordinates,  substituting \eqref{eq_4_19}, and using that $(e_1+i\eta_2)\cdot \nabla a_0^{(1)}=i$,  we obtain that 
\begin{equation}
\label{eq_4_20}
\begin{aligned}
&\sum_{i_1,i_2=1}^n \int_{\R^n} B^{(2)}_{i_1 i_2}(x_1,y_2,y'') (e_1+i\eta_2)_{i_1}(e_1+i\eta_2)_{i_2} \\
&\quad \quad \quad \quad \quad (y_2-c)^{2+l} e^{-i\lambda(x_1+iy_2)}g(y'')dx_1dy_2dy''\\
&+2
\sum_{i_1=1}^n  \int_{\R^n}  W^{(1)}_{i_1}(x_1,y_2,y'')(e_1+i\eta_2)_{i_1}(y_2-c)^{1+l} e^{-i\lambda(x_1+iy_2)}g(y'')dx_1dy_2dy''=0.
\end{aligned}
\end{equation}
Since $B^{(2)}\in C_0(\R^n; S^2)$, $W^{(1)}\in C_0(\R^n; S^1)$, and considering that $g\in C^\infty(\R^{n-2})$ is arbitrary, we can argue as above and deduce from \eqref{eq_4_20} that
\begin{equation}
\label{eq_4_21}
\begin{aligned}
&\sum_{i_1,i_2=1}^n \int_{\R} t^{2+l} \widehat {B^{(2)}_{i_1 i_2}}(\lambda,y_0'+t\eta_2') (e_1+i\eta_2)_{i_1}(e_1+i\eta_2)_{i_2} e^{\lambda t }dt\\
&+2
\sum_{i_1=1}^n  \int_{\R} t^{1+l} \widehat{W^{(1)}_{i_1}}(\lambda,y_0'+t\eta_2')  (e_1+i\eta_2)_{i_1} e^{\lambda t }dt=0,
\end{aligned}
\end{equation}
for all $y_0'=(c,y'')\in \R^{n-1}$, all $\lambda\in \R$, and all $l=0,1$.  Here we have made the change of variables $t=y_2-c$. 

Letting $\lambda=0$, \eqref{eq_4_21} becomes 
\begin{equation}
\label{eq_4_22}
\begin{aligned}
&\sum_{i_1,i_2=1}^n \int_{\R} t^{2+l} \widehat {B^{(2)}_{i_1 i_2}}(0,y_0'+t\eta_2') (e_1+i\eta_2)_{i_1}(e_1+i\eta_2)_{i_2} dt\\
&+2
\sum_{i_1=1}^n  \int_{\R} t^{1+l} \widehat{W^{(1)}_{i_1}}(0,y_0'+t\eta_2')  (e_1+i\eta_2)_{i_1} dt=0,
\end{aligned}
\end{equation}
for  all $y_0'\in \R^{n-1}$, $l=0,1$, and all $\eta_2\in \R^n$ such that $\eta_2\cdot e_1=0$ and $|\eta_2|=1$. Here $\hat {B^{(2)}}(0,\cdot)\in C_0(\R^{n-1}; S^2(\R^n))$ is such that $j_\delta (\hat{B^{(2)}}(0,\cdot))=0$ thanks to  \eqref{eq_4_3_decom_new_2_for_2_tensors}, and $\hat{W^{(1)}}(0,\cdot)\in C_0(\R^{n-1}; S^1(\R^n))$.

Applying  Corollary \ref{cor_inversion_2} to \eqref{eq_4_22}, we conclude that 
\begin{equation}
\label{eq_4_23}
\widehat {B^{(2)}}(0,\cdot)=0, \quad \widehat{W^{(1)}}(0,\cdot)=0. 
\end{equation}
Differentiating \eqref{eq_4_21} with respect to $\lambda$, letting $\lambda=0$, and using \eqref{eq_4_23}, we get 
\begin{equation}
\label{eq_4_24}
\begin{aligned}
&\sum_{i_1,i_2=1}^n \int_{\R} t^{2+l} \p_\lambda \widehat {B^{(2)}_{i_1 i_2}}(0,y_0'+t\eta_2') (e_1+i\eta_2)_{i_1}(e_1+i\eta_2)_{i_2} dt\\
&+2
\sum_{i_1=1}^n  \int_{\R} t^{1+l}\p_\lambda \widehat{W^{(1)}_{i_1}}(0,y_0'+t\eta_2')  (e_1+i\eta_2)_{i_1} dt=0,
\end{aligned}
\end{equation}
for  all $y_0'\in \R^{n-1}$, $l=0,1$, and all $\eta_2\in \R^n$ such that $\eta_2\cdot e_1=0$ and $|\eta_2|=1$. Here $ \p_\lambda  \hat {B^{(2)}}(0,\cdot)\in C_0(\R^{n-1}; S^2(\R^n))$ is such that $j_\delta ( \p_\lambda  \hat{B^{(2)}}(0,\cdot))=0$ thanks to  \eqref{eq_4_3_decom_new_2_for_2_tensors}, and $ \p_\lambda  \hat{W^{(1)}}(0,\cdot)\in C_0(\R^{n-1}; S^1(\R^n))$.  An application of  Corollary \ref{cor_inversion_2} to \eqref{eq_4_24} shows that 
$\p_\lambda \widehat {B^{(2)}}(0,\cdot)=0$, and $\p_\lambda\widehat{W^{(1)}}(0,\cdot)=0$. Continuing in the same fashion, we see that  $\p_\lambda^k \widehat {B^{(2)}}(0,\cdot)=0$ and $\p_\lambda^k\widehat{W^{(1)}}(0,\cdot)=0$, for every $k=0,1,2,\dots$.

As $B^{(2)}(x_1,\cdot)$ and $W^{(1)}(x_1,\cdot)$ are compactly supported in $x_1$, the functions $\lambda\mapsto \widehat {B^{(2)}}(\lambda,\cdot)$ and $\lambda\mapsto \widehat{W^{(1)}}(\lambda,\cdot)$ are analytic. Thus, we get $\widehat{B^{(2)}}(\lambda,\cdot)=0$ and $\widehat{W^{(1)}}(\lambda,\cdot)=0$ for all $\lambda\in\R$. Hence, $B^{(2)}=0$ and $W^{(1)}=0$. Therefore, it follows from \eqref{eq_4_15} that $A^{(3)}=0$, completing the recovery of $A^{(3)}$. We also conclude from \eqref{eq_4_3_decom_new_for_2_tensors} that
\begin{equation}
\label{eq_4_28}
A^{(2)}=i_\delta W^{(0)}.
\end{equation}

Let us now return to the integral identity \eqref{int_eq_6} with $A^{(3)}=0$ and $A^{(2)}$ from \eqref{eq_4_28}. Now it would be convenient to work with three different biharmonic functions $v^{(0)}, v^{(1)}, v^{(2)}\in C^5(\overline{\Omega})$ given by 
\begin{equation}
\label{eq_4_28_1}
\begin{aligned}
v^{(0)}&=e^{\frac{-2x\cdot (e_1+i\eta_2)}{h}}(a_0^{(0)}(x)+r^{(0)}(x; h)),\\
v^{(1)}&=e^{\frac{x\cdot (e_1+i\eta_2)}{h}}(a_0^{(1)}(x)+r^{(1)}(x; h)),\quad v^{(2)}=e^{\frac{x\cdot (e_1+i\eta_2)}{h}}(a_0^{(2)}(x)+r^{(2)}(x; h)), 
\end{aligned}
\end{equation}
where $a_0^{(l)}\in C^\infty(\R^n)$ satisfying \eqref{eq_3_9},  and 
\begin{equation}
\label{eq_4_28_2}
\|r^{(l)}\|_{C^5(\overline{\Omega})}=\mathcal{O}(h),
\end{equation}
as $h\to 0$,  $l=0,1, 2$.  

We will now substitute $v^{(0)}, v^{(1)}, v^{(2)}$ given by \eqref{eq_4_28_1} into the integral identity \eqref{int_eq_6} with $A^{(3)}=0$ and $A^{(2)}$ from \eqref{eq_4_28}. By utilizing \eqref{eq_4_3_decom_new_2_for_2_tensors_2} and \eqref{eq_4_28_2}, multiplying \eqref{int_eq_6} by $h$ and letting $h$ tend to $0$, we obtain that 
\begin{equation}
\label{eq_4_30}
\begin{aligned}
&\sum_{i_1=1}^n \int_{\Omega} A^{(1)}_{i_1}(x) (e_1+i\eta_2)_{i_1}a_0^{(1)}(x) a_0^{(2)}(x) a_0^{(0)}(x)dx\\
&+\frac{1}{i}\sum_{i_1,i_2=1}^n \int_{\Omega} (i_\delta W^{(0)})_{i_1 i_2 }(x) (e_1+i\eta_2)_{i_1}\p_{i_2} (a_0^{(1)}(x) a_0^{(2)}(x)) a_0^{(0)}(x)dx=0.
\end{aligned}
\end{equation}
By using \eqref{eq_5_3} and \eqref{eq_5_2}, we can deduce from \eqref{eq_4_30} that 
\begin{equation}
\label{eq_4_31}
\begin{aligned}
&\sum_{i_1=1}^n \int_{\R^n} A^{(1)}_{i_1}(x) (e_1+i\eta_2)_{i_1}a_0^{(1)}(x)a_0^{(2)}(x) a_0^{(0)}(x)dx\\
&+\frac{1}{i}\int_{\R^n} W^{(0)}(x)  a_0^{(0)}(x) (e_1+i\eta_2)\cdot \nabla (a_0^{(1)}(x)a_0^{(2)}(x)) dx=0.
\end{aligned}
\end{equation}

As before, we shall consider an orthonormal basis $\{e_1,\eta_2,\eta_3, \dots, \eta_n\}$ in $\R^n$, and denote the coordinates in this basis by $y=(x_1,y_2,y'')$, $y''=(y_3,\dots, y_n)$.  According to Remark \ref{rem_choice_amplitude}, we choice 
\begin{equation}
\label{eq_4_32}
a_0^{(1)}(x_1,y_2,y'')=a_0^{(2)}(x_1,y_2,y'')=1, \quad a_0^{(0)}(x_1,t,y'')=(y_2-c) e^{-i\lambda(x_1+iy_2)}g(y''), 
\end{equation}
where $c\in\R$,  $\lambda\in \R$, and $g\in C^\infty(\R^{n-2})$. 

Writing out the integral in  \eqref{eq_4_31}  in $y$--coordinates, and substituting \eqref{eq_4_32}, we get 
\begin{equation}
\label{eq_4_33}
\sum_{i_1=1}^n \int_{\R^n} A^{(1)}_{i_1}(x_1, y_1,y'') (e_1+i\eta_2)_{i_1}(y_2-c) e^{-i\lambda(x_1+iy_2)}g(y'')dx_1dy_2dy''=0.
\end{equation}
By considering that $A^{(1)}\in C_0(\R^n; S^1)$ and the fact that $g\in C^\infty(\R^{n-2})$ is an arbitrary function, and arguing as before, we can deduce from \eqref{eq_4_33} that 
\begin{equation}
\label{eq_4_34}
\sum_{i_1=1}^n \int_{\R} t \widehat{A^{(1)}_{i_1}}(\lambda, y_0'+t\eta_2') (e_1+i\eta_2)_{i_1} e^{\lambda t}dt=0,
\end{equation}
for all $y_0'=(c,y'')\in \R^{n-1}$ and all $\lambda\in \R$. Letting $\lambda=0$, \eqref{eq_4_34} becomes 
\begin{equation}
\label{eq_4_35}
\sum_{i_1=1}^n \int_{\R} t \widehat {A^{(1)}_{i_1}}(0, y_0'+t\eta_2') (e_1+i\eta_2)_{i_1} dt=0,
\end{equation}
for all $y_0'\in \R^{n-1}$, and all $\eta_2\in \R^n$ such that $\eta_2\cdot e_1=0$ and $|\eta_2|=1$.  Here $\widehat {A^{(1)}}(0,\cdot)\in C_0(\R^{n-1}; S^1(\R^n))$ and therefore, an application of Lemma \ref{lem_inversion_1} shows that $\widehat {A^{(1)}}(0,\cdot)=0$.  Proceeding as before by differentiating \eqref{eq_4_34} with respect to $\lambda$, letting $\lambda=0$, applying Lemma \ref{lem_inversion_1}, and inverting the Fourier transform,  we conclude that  $A^{(1)}=0$.

Inserting $A^{(1)}=0$ in \eqref{eq_4_31}, we get
\begin{equation}
\label{eq_4_36}
\int_{\R^n} W^{(0)}(x) a_0^{(0)}(x) (e_1+i\eta_2)\cdot \nabla (a_0^{(1)}(x)a_0^{(2)}(x)) dx=0.
\end{equation}
Writing out the integral in  \eqref{eq_4_36}  in $y$--coordinates,  and choosing  
\[
a_0^{(1)}(x_1,y_2,y'')=y_2, \ a_0^{(2)}(x_1,y_2,y'')=1,\  a_0^{(0)}(x_1,t,y'')= e^{-i\lambda(x_1+iy_2)}g(y''), 
\]
where $\lambda\in \R$ and $g\in C^\infty(\R^{n-2})$, we obtain from \eqref{eq_4_36} that 
\begin{equation}
\label{eq_4_37}
 \int_{\R^n} W^{(0)}(x_1,y_2,y'') e^{-i\lambda(x_1+iy_2)}g(y'')dx_1dy_2dy''=0.
\end{equation}

Since $W^{(0)}\in C_0(\R^n;S^0)$,  and $g\in C^\infty(\R^{n-2})$ is arbitrary, proceeding as previously discussed, we can deduce from \eqref{eq_4_37} that 
\begin{equation}
\label{eq_4_38}
 \int_{\R} \widehat{W^{(0)}}(\lambda,y_0'+t\eta_2') e^{\lambda t}dt=0.
\end{equation}
This holds for all $y_0'=(c,y'')\in \R^{n-1}$, all $\eta_2'\in \R^{n-1}$ such that $|\eta_2'|=1$, and all $\lambda\in \R$. We have $\p_\lambda^k\widehat{W^{(0)}}(0,\cdot)\in C_0(\R^{n-1}; S^0)$, $k=0,1,\dots$.

When setting $\lambda$ to zero in  \eqref{eq_4_38}, we obtain  $(J^{0,0}\widehat{W^{(0)}}(0,\cdot))(y_0',\eta_2')=0$ for all $(y_0,\eta_2')\in \R^{n-1}\times \mathbb{S}^{n-2}$. As a consequence of Theorem \ref{thm_Bhattacharyya_Krishnan_Sahoo}, we can conclude that $\widehat{W^{(0)}}(0,\cdot)=0$. Following the same strategy as before, which involves differentiating \eqref{eq_4_38} with respect to $\lambda$, setting $\lambda$ to zero, utilizing Theorem \ref{thm_Bhattacharyya_Krishnan_Sahoo}, and inverting the Fourier transform, we conclude from \eqref{eq_4_38} that $W^{(0)}=0$. Thus, it follows from \eqref{eq_4_28} that $A^{(2)}=0$. 

Let us return to integral identity \eqref{int_eq_6} while setting $A^{(1)}=0$, $A^{(2)}=0$, and $A^{(3)}=0$, and taking $v^{(0)}=1$. We have
\begin{equation}
\label{eq_4_41}
\int_{\Omega} qv^{(1)}v^{(2)} \,dx = 0,
\end{equation}
for all $v^{(1)}, v^{(2)} \in C^{4,\alpha}(\overline{\Omega})$ solving $(-\Delta)^2 v^{(l)}=0$ in $\Omega$, $l=1,2$. To conclude that $q=0$, we follow a standard argument, see \cite{Isakov_1991, Krup_Lassas_Uhlmann_2014}. Let $\xi\in\R^n$ be arbitrary, and let $\alpha,\beta\in \R^n$ be such that $|\alpha|=|\beta|=1$ and $\xi\cdot\alpha=\xi\cdot\beta=\alpha\cdot\beta=0$. Utilizing Proposition \ref{prop_CGO_q},  we take
\begin{equation}
\label{eq_4_42}
v^{(1)}=e^{\frac{x\cdot(\alpha+i\beta)}{h}}(e^{-ix\cdot \xi}+ r^{(1)}(x; h)), \quad v^{(2)}=e^{\frac{-x\cdot(\alpha+i\beta)}{h}}(1+ r^{(2)}(x; h)),
\end{equation}
where $\|r^{(l)}\|_{C^5(\overline{\Omega})}=\mathcal{O}(h)$ as $h\to 0$ for $l=1, 2$. Substituting $v^{(1)}, v^{(2)}$ given by \eqref{eq_4_42} into the integral identity \eqref{eq_4_41} and letting $h\to 0$, we obtain
\[
\int_{\Omega} q e^{-ix\cdot \xi} \,dx = 0,
\]
for all $\xi\in\R^n$. Thus, $q=0$ in $\Omega$. This completes the proof of Theorem \ref{thm_density}.

\section{Proof of Theorem \ref{thm_main}}

\label{sec_thm_main}

In this section, it will be convenient for us to write $A^{(1), (j)}:=A^{(j)}$, $q^{(1)}:=q$, and  $A^{(2), (j)}:=\tilde A^{(j)}$, $q^{(2)}:=\tilde q$, $j=1,2,3$.  Let $\varepsilon=(\varepsilon_1,\dots, \varepsilon_m)\in \C^m$, $m\ge 2$, and consider the Dirichlet problem \eqref{int_eq_4} for $L_{A^{(r),(1)}, A^{(r),(2)}, A^{(r),(3)}, q^{(r)}}$ with 
\[
f=\sum_{k=1}^m\varepsilon_k f_k, \quad g=\sum_{k=1}^m\varepsilon_k g_k, \quad f_k \in C^{4,\alpha}(\p \Omega), \quad g_k\in C^{3,\alpha}(\p \Omega), \quad k=1,\dots, m,
\]
where $r=1,2$. 
By Theorem \ref{thm_well-posedness}, for all $|\varepsilon|$ sufficiently small, the problem \eqref{int_eq_4} has a unique small solution $u_r(\cdot, \varepsilon)\in C^{4,\alpha}(\overline{\Omega})$, which depends holomorphically on $\varepsilon\in \text{neigh}(0, \C^m)$.  We shall use an induction argument on $m\ge 2$ to prove that the equality  $\Lambda_{A^{(1), (1)},A^{(1),(2)},A^{(1), (3)},q^{(1)}}=\Lambda_{A^{(2), (1)},A^{(2),(2)},A^{(2), (3)},q^{(2)}}$ implies that 
\[
A^{(1), (j),{m-1}}=A^{(2), (j),{m-1}}, \ j=1,2,3, \quad q^{(1), m}=q^{(2), m} \ \text{in}\ \overline{\Omega}\times\C, \quad m=2,3,\dots, 
\]
where $A^{(r), (j),{m-1}}$ are given in  \eqref{int_eq_2}, and $q^{(r), m}$  are given in  \eqref{int_eq_3}. 

First, let $m=2$, and let us start by performing a second-order linearization of the Dirichlet--to--Neumann map. Let $u_r=u_r(x,\varepsilon)$ be the unique small solution to the Dirichlet problem 
\begin{equation}
\label{eq_2_1}
\begin{cases}
(-\Delta)^2u_r+\sum_{j=1}^3 \sum_{i_1,\dots, i_j=1}^n \big( \sum_{k=1}^\infty A^{(r),(j), k}_{i_1\dots i_j}(x)\frac{u_r^k}{k!} \big) D^j_{i_1\dots i_j}u_r\\
\quad \quad \quad \quad \quad \quad \quad+  \sum_{k=2}^\infty q^{(r),k}(x)\frac{u_r^k}{k!}=0\quad \text{in}\quad \Omega,\\
u_r|_{\p \Omega}=\varepsilon_1f_1+\varepsilon_2 f_2,\\
\p_\nu u_r|_{\p \Omega}=\varepsilon_1g_1+\varepsilon_2 g_2,
\end{cases}
\end{equation}
for $r=1,2$.  Applying $\p_{\varepsilon_l}|_{\varepsilon=0}$ to \eqref{eq_2_1}, $l=1,2$, and using that $u_r(x,0)=0$, we get 
\begin{equation}
\label{eq_2_2}
\begin{cases}
(-\Delta)^2v_r^{(l)}=0 \quad \text{in}\quad \Omega,\\
v_r^{(l)}|_{\p \Omega}=f_l,\\
\p_\nu v_r^{(l)}|_{\p \Omega}=g_l,
\end{cases}
\end{equation}
where $v_r^{(l)}=\p_{\varepsilon_l} u_r|_{\varepsilon=0}$, $r=1,2$. By the uniqueness for the Dirichlet problem  \eqref{eq_2_2}, we see that $v^{(l)}:=v_1^{(l)}=v_2^{(l)}\in C^{4,\alpha}(\overline{\Omega})$, $l=1,2$, see  \cite[Theorem 2.19]{Gazzola_Grunau_Sweers_book_2010}, \cite[Section 6.6]{Evans_book}. 

Applying $\p_{\varepsilon_1}\p_{\varepsilon_2}|_{\varepsilon=0}$ to \eqref{eq_2_1}, we get the second order linearization, 
\begin{equation}
\label{eq_2_3}
\begin{cases}
(-\Delta)^2w_r+\sum_{j=1}^3 \sum_{i_1,\dots, i_j=1}^n  A^{(r),(j), 1}_{i_1\dots i_j}(x)\big(v^{(1)} D^j_{i_1\dots i_j}v^{(2)}+v^{(2)} D^j_{i_1\dots i_j}v^{(1)}\big)\\
\quad\quad \quad \quad \quad \quad \quad\quad \quad \quad \quad  + q^{(r), 2}v^{(1)}v^{(2)}=0\quad \text{in}\quad \Omega,\\
w_r|_{\p \Omega}=0,\\
\p_\nu w_r|_{\p \Omega}=0,
\end{cases}
\end{equation}
where $w_r=\p_{\varepsilon_1}\p_{\varepsilon_2}u_r|_{\varepsilon=0}$,  $r=1,2$.  
Now the equality 
$
\Lambda_{A^{(1),(1)},A^{(1),(2)},A^{(1),(3)},q^{(1)}}(\varepsilon_1f_1+\varepsilon_2 f_2, \varepsilon_1g_1+\varepsilon_2 g_2)=\Lambda_{A^{(2),(1)},A^{(2),(2)},A^{(2),(3)},q^{(2)}}(\varepsilon_1f_1+\varepsilon_2 f_2, \varepsilon_1g_1+\varepsilon_2 g_2)$,
for all small $\varepsilon$ and all $f_k \in C^{4,\alpha}(\p \Omega)$, $g_k\in C^{3,\alpha}(\p \Omega)$, $k=1,2$, implies that $\p_\nu^2 u_1|_{\p \Omega}=\p_\nu^2 u_2|_{\p \Omega}$ and 
$\p_\nu^3 u_1|_{\p \Omega}=\p_\nu^3 u_2|_{\p \Omega}$. Hence, applying $\p_{\varepsilon_1}\p_{\varepsilon_2}|_{\varepsilon=0}$, we conclude that $\p_\nu^2 w_1|_{\p \Omega}=\p_\nu^2 w_2|_{\p \Omega}$ and 
$\p_\nu^3 w_1|_{\p \Omega}=\p_\nu^3 w_2|_{\p \Omega}$.

It follows from \eqref{eq_2_3} that 
\begin{equation}
\label{eq_2_4}
\begin{aligned}
(-\Delta)^2(&w_1-w_2)\\
&=\sum_{j=1}^3 \sum_{i_1,\dots, i_j=1}^n  (A^{(2),(j), 1}_{i_1\dots i_j}(x)-A^{(1),(j), 1}_{i_1\dots i_j}(x))\big(v^{(1)} D^j_{i_1\dots i_j}v^{(2)}+v^{(2)} D^j_{i_1\dots i_j}v^{(1)}\big)\\  
&+ (q^{(2), 2}-q^{(1), 2})v^{(1)}v^{(2)}=0\quad \text{in}\quad \Omega. 
\end{aligned}
\end{equation}
We also have $(w_1-w_2)|_{\p \Omega}=0$, $\p_\nu (w_1-w_2)|_{\p \Omega}=0$, $\p_\nu^2 (w_1-w_2)|_{\p \Omega}=0$, and $\p_\nu^3 (w_1-w_2)|_{\p \Omega}=0$. Multiplying \eqref{eq_2_4} by $v^{(0)}\in C^{4,\alpha}(\overline{\Omega})$ such that $(-\Delta)^2 v^{(0)}=0$, integrating over $\Omega$,  and using Green's formula, see \cite[Theorem 10.1]{Agmon_book_1965}, we obtain that 
\begin{equation}
\label{eq_2_5}
\begin{aligned}
&\sum_{j=1}^3 \sum_{i_1,\dots, i_j=1}^n  \int_\Omega (A^{(2),(j), 1}_{i_1\dots i_j}(x)-A^{(1),(j), 1}_{i_1\dots i_j}(x))\big(v^{(1)} D^j_{i_1\dots i_j}v^{(2)}+v^{(2)} D^j_{i_1\dots i_j}v^{(1)}\big)v^{(0)}dx\\  
&+ \int_{\Omega}(q^{(2), 2}-q^{(1), 2})v^{(1)}v^{(2)}v^{(0)}dx=0. 
\end{aligned}
\end{equation}
Note that \eqref{eq_2_5} holds for all $v^{(l)}\in C^{4,\alpha}(\overline{\Omega})$ such that $(-\Delta)^2 v^{(l)}=0$ in $\Omega$, $l=1,2,3$. In view of \eqref{eq_int_boundary_determination}, an application of Theorem  \ref{thm_density} allows us to conclude that 
\[
A^{(1),(j), 1}=A^{(2),(j), 1}, \quad j=1,2,3, \quad q^{(1),2}=q^{(2),2}. 
\]

Let $m\ge 3$ and let us assume that 
\begin{equation}
\label{eq_2_6}
A^{(1),(j), k}=A^{(2),(j),k}, \quad j=1,2,3, \quad q^{(1),k+1}=q^{(2),k+1}, \quad k=1,\dots, m-2. 
\end{equation}
To prove that 
\[
A^{(1),(j), m-1}=A^{(2),(j),m-1}, \quad j=1,2,3, \quad q^{(1),m}=q^{(2),m}, 
\]
we shall carry out the $m$th order linearization of the Dirichlet--to--Neumann map. In doing so, we let $u_r=u_r(x,\varepsilon)$ be the unique small solution to 
\begin{equation}
\label{eq_2_7}
\begin{cases}
(-\Delta)^2u_r+\sum_{j=1}^3 \sum_{i_1,\dots, i_j=1}^n \big( \sum_{k=1}^\infty A^{(r),(j), k}_{i_1\dots i_j}(x)\frac{u_r^k}{k!} \big) D^j_{i_1\dots i_j}u_r\\
\quad \quad \quad \quad \quad \quad \quad+ \sum_{k=2}^\infty q^{(r),k}(x)\frac{u_r^k}{k!}=0\quad \text{in}\quad \Omega,\\
u_r|_{\p \Omega}=\varepsilon_1f_1+\dots+ \varepsilon_m f_m,\\
\p_\nu u_r|_{\p \Omega}=\varepsilon_1g_1+\dots+ \varepsilon_m g_m,
\end{cases}
\end{equation}
for $r=1,2$.  We would like to apply $\p_{\varepsilon_1}\dots\p_{\varepsilon_m}|_{\varepsilon=0}$ to \eqref{eq_2_7}. To that end, we first note that 
\[
\p_{\varepsilon_1}\dots\p_{\varepsilon_m} \bigg(\sum_{j=1}^3 \sum_{i_1,\dots, i_j=1}^n \big( \sum_{k=m}^\infty A^{(r),(j), k}_{i_1\dots i_j}(x)\frac{u_r^k}{k!} \big) D^j_{i_1\dots i_j}u_r
+ \sum_{k=m+1}^\infty q^{(r),k}(x)\frac{u_r^k}{k!}\bigg)
\] 
is a sum of terms, each of them containing positive powers of $u_r$ which vanish when $\varepsilon=0$. The only term in the expression for $\p_{\varepsilon_1}\dots\p_{\varepsilon_m}\big(q^{(r),m}(x)\frac{u_r^m}{m!}\big)$ which does not contain a positive power of $u_r$ is given by 
\begin{equation}
\label{eq_2_8}
q^{(r),m}(x) \p_{\varepsilon_1}u_r\dots\p_{\varepsilon_m}u_r.
\end{equation}
 The only term in 
$\p_{\varepsilon_1}\dots\p_{\varepsilon_m}\big(\sum_{j=1}^3 \sum_{i_1,\dots, i_j=1}^n  A^{(r),(j), m-1}_{i_1\dots i_j}(x)\frac{u_r^{m-1}}{(m-1)!} D^j_{i_1\dots i_j}u_r\big)$ which does not contain a power of $u_r$ is 
\begin{equation}
\label{eq_2_8_1}
\sum_{j=1}^3 \sum_{i_1,\dots, i_j=1}^n  A^{(r),(j), m-1}_{i_1\dots i_j}(x)\bigg(\sum_{k=1}^m \prod_{l=1,l\ne k}^m \p_{\varepsilon_l}u_r D^j_{i_1\dots i_j}\p_{\varepsilon_k}u_r\bigg).
\end{equation}
Finally, the expression 
\begin{equation}
\label{eq_2_9}
\p_{\varepsilon_1}\dots\p_{\varepsilon_m} \bigg(\sum_{j=1}^3 \sum_{i_1,\dots, i_j=1}^n \big( \sum_{k=1}^{m-2} A^{(r),(j), k}_{i_1\dots i_j}(x)\frac{u_r^k}{k!} \big) D^j_{i_1\dots i_j}u_r
+ \sum_{k=2}^{m-1}q^{(r),k}(x)\frac{u_r^k}{k!}\bigg)\bigg|_{\varepsilon=0}
\end{equation}
is independent of $r=1,2$. Indeed, this follows from \eqref{eq_2_6}, the fact that \eqref{eq_2_9} contains only the derivatives of $u_r$ of the form $\p^s_{\varepsilon_{l_1}, \dots, \varepsilon_{l_s}}u_r|_{\varepsilon=0}$ with $s=1,\dots, m-1$, $\varepsilon_{l_1}, \dots, \varepsilon_{l_s}\in \{\varepsilon_1,\dots, \varepsilon_m\}$, and the fact that 
\[
\p^s_{\varepsilon_{l_1}, \dots, \varepsilon_{l_s}}u_1|_{\varepsilon=0}=\p^s_{\varepsilon_{l_1}, \dots, \varepsilon_{l_s}}u_2|_{\varepsilon=0}.
\]
The latter can be established by induction on $s$, applying the operator  $\p^s_{\varepsilon_{l_1}, \dots, \varepsilon_{l_s}}|_{\varepsilon=0}$ to \eqref{eq_2_7} and using \eqref{eq_2_6} together with the unique solvability of the Dirichlet problem for the operator $(-\Delta)^2$. 

Thus,  applying $\p_{\varepsilon_1}\dots\p_{\varepsilon_m}|_{\varepsilon=0}$ to \eqref{eq_2_7}, we get 
\begin{equation}
\label{eq_2_10}
\begin{cases}
(-\Delta)^2 w_r+\sum_{j=1}^3 \sum_{i_1,\dots, i_j=1}^n  A^{(r),(j), m-1}_{i_1\dots i_j}(x)\bigg(\sum_{k=1}^m \prod_{l=1,l\ne k}^m v^{(l)} D^j_{i_1\dots i_j}v^{(k)}\bigg)
\\
\quad \quad \quad \quad \quad \quad \quad+  q^{(r),m}(x)v^{(1)}\dots v^{(m)}=H_m\quad \text{in}\quad \Omega,\\
w_r|_{\p \Omega}=0,\\
\p_\nu w_r|_{\p \Omega}=0,
\end{cases}
\end{equation}
for $r=1,2$, see \eqref{eq_2_8}, \eqref{eq_2_8_1}. Here $w_r=\p_{\varepsilon_1}\dots\p_{\varepsilon_m}u_r|_{\varepsilon=0}$ and 
\begin{align*}
H_m(x):=-&\p_{\varepsilon_1}\dots\p_{\varepsilon_m} \\
&\bigg(\sum_{j=1}^3 \sum_{i_1,\dots, i_j=1}^n \big( \sum_{k=1}^{m-2} A^{(r),(j), k}_{i_1\dots i_j}(x)\frac{u_r^k}{k!} \big) D^j_{i_1\dots i_j}u_r
+  \sum_{k=2}^{m-1}q^{(r),k}(x)\frac{u_r^k}{k!}\bigg)\bigg|_{\varepsilon=0}.
\end{align*}
Letting $w=w_1-w_2$, we obtain from \eqref{eq_2_10} that 
\begin{equation}
\label{eq_2_11}
\begin{cases}
(-\Delta)^2 (w_1-w_2) 
=\sum_{j=1}^3 \sum_{i_1,\dots, i_j=1}^n  (A^{(2),(j), m-1}_{i_1\dots i_j}(x)-A^{(1),(j), m-1}_{i_1\dots i_j}(x))\\\quad \quad \quad \quad \quad \quad \quad \bigg(\sum_{k=1}^m \prod_{l=1,l\ne k}^m v^{(l)} D^j_{i_1\dots i_j}v^{(k)}\bigg)\\
\quad \quad \quad \quad \quad \quad \quad+  (q^{(2),m}(x)- q^{(1),m}(x))v^{(1)}\dots v^{(m)}\quad \text{in}\quad \Omega,\\
w|_{\p \Omega}=0,\\
\p_\nu w|_{\p \Omega}=0.
\end{cases}
\end{equation}
 Applying $\p_{\varepsilon_1}\dots\p_{\varepsilon_m}|_{\varepsilon=0}$ to 
$
\Lambda_{A^{(1),(1)},A^{(1),(2)},A^{(1),(3)},q^{(1)}}(\varepsilon_1f_1+\dots+ \varepsilon_m f_m, \varepsilon_1g_1+\dots+ \varepsilon_m g_m)=\Lambda_{A^{(2),(1)},A^{(2),(2)},A^{(2),(3)},q^{(2)}}(\varepsilon_1f_1+\dots+ \varepsilon_m f_m, \varepsilon_1g_1+\dots+ \varepsilon_m g_m)$,
we get $\p_\nu^2 w|_{\p \Omega}=0$ and 
$\p_\nu^3 w=0$. Let $v^{(0)}\in C^{4,\alpha}(\overline{\Omega})$ be such that $(-\Delta)^2v^{(0)}=0$ in $\Omega$. Multiplying \eqref{eq_2_11} by $v^{(0)}$ and using Green's formula, see \cite[Theorem 10.1]{Agmon_book_1965},  we obtain that 
\begin{equation}
\label{eq_2_12}
\begin{aligned}
\sum_{j=1}^3 \sum_{i_1,\dots, i_j=1}^n \int_{\Omega} (A^{(2),(j), m-1}_{i_1\dots i_j}(x)-A^{(1),(j), m-1}_{i_1\dots i_j}(x))\bigg(\sum_{k=1}^m \prod_{l=1,l\ne k}^m v^{(l)} D^j_{i_1\dots i_j}v^{(k)}\bigg) v^{(0)}dx \\
+ \int_{\Omega} (q^{(2),m}(x)- q^{(1),m}(x))v^{(1)}\dots v^{(m)}v^{(0)}dx=0,
\end{aligned}
\end{equation}
for all $v^{(l)}\in C^{4,\alpha}(\overline{\Omega})$ solving $(-\Delta)^2v^{(l)}=0$ in $\Omega$, $l=0,1,\dots, m$. 
 
Letting $v^{(3)}=\dots=v^{(m)}=1$ in \eqref{eq_2_12}, we get 
\begin{equation}
\label{eq_2_13}
\begin{aligned}
\sum_{j=1}^3 \sum_{i_1,\dots, i_j=1}^n \int_{\Omega} (A^{(2),(j), m-1}_{i_1\dots i_j}(x)-&A^{(1),(j), m-1}_{i_1\dots i_j}(x))\big( v^{(1)} D^j_{i_1\dots i_j}v^{(2)}+ v^{(2)} D^j_{i_1\dots i_j}v^{(1)}\big) 
\\
&v^{(0)}dx 
+ \int_{\Omega} (q^{(2),m}(x)- q^{(1),m}(x))v^{(1)}v^{(2)}v^{(0)}dx=0,
\end{aligned}
\end{equation}
 for all $v^{(l)}\in C^{4,\alpha}(\overline{\Omega})$ solving $(-\Delta)^2v^{(l)}=0$ in $\Omega$, $l=0,1,2$. 
In view of  \eqref{eq_int_boundary_determination}, applying Theorem  \ref{thm_density}, we conclude that 
\[
A^{(1),(j), m-1}=A^{(2),(j), m-1}, \quad j=1,2,3, \quad q^{(1),m}=q^{(2),m}. 
\]
This completes the proof of Theorem \ref{thm_main}.

\begin{appendix}

\section{Well-posedness of the Dirichlet problem for a third-order nonlinear perturbation of a biharmonic operator}

\label{sec_well-posedness}

In this appendix, we shall show the well-posedness of the Dirichlet problem for a third-order nonlinear perturbation of a biharmonic operator. The argument is standard, see \cite{Krup_Uhlmann_2020}, 
\cite{Lassas_Liimatainen_Lin_Salo_2021}, and is presented here for completeness and convenience of the reader.

Let $\Omega\subset \R^n$, $n\ge 3$, be a bounded open set with $C^\infty$ boundary. Let $k\in \N\cup \{0\}$, $0<\alpha<1$, and let $C^{k,\alpha}(\overline{\Omega})$ be the standard H\"older space, see \cite{Hormander_1976}. Note that $C^{k,\alpha}(\overline{\Omega})$ is an algebra under pointwise multiplication, with 
\begin{equation}
\label{eq_8_1}
\|uv\|_{C^{k,\alpha}(\overline{\Omega})}\le C\big(\|u\|_{C^{k,\alpha}(\overline{\Omega})}\|v\|_{L^\infty(\Omega)}+\|u\|_{L^\infty(\Omega)} \|v\|_{C^{k,\alpha}(\overline{\Omega})}\big), \quad u,v\in C^{k,\alpha}(\overline{\Omega}), 
\end{equation}
see \cite[Theorem A.7]{Hormander_1976}. 

Consider the Dirichlet problem for the operator $L_{A^{(1)},A^{(2)},A^{(3)},q}$, given by  \eqref{int_eq_1},
\begin{equation}
\label{eq_8_2}
\begin{cases}
L_{A^{(1)},A^{(2)},A^{(3)},q}u=0\quad \text{in}\quad \Omega,\\
u|_{\p \Omega}=f,\\
\p_\nu u|_{\p \Omega}=g,
\end{cases}
\end{equation}
where $\nu$ is the unit outer normal to the boundary. Here 
$A^{(j)}: \overline{\Omega}\times \C\to S^j$, $j=1,2,3$, and $q:\overline{\Omega}\times\C\to \C$ satisfy the conditions: 
\begin{itemize}
\item[(i)] the map $\C\ni z\mapsto A^{(j)}(\cdot,z)$ is holomorphic with values in $C^{0,\alpha}(\overline{\Omega};  S^j)$ for some $0<\alpha<1$, the space of H\"older continuous symmetric $j$--tensor fields, 

\item[(ii)] the map $\C\ni z\mapsto q(\cdot,z)$ is holomorphic with values in $C^{0,\alpha}(\overline{\Omega})$,

\item[(iii)] $A^{(j)}(x,0)=0$, $q(x,0)=0$, and $\p_z q(x,0)=0$ for all $x\in \overline{\Omega}$, $j=1,2,3$. 
\end{itemize}

It follows from (i), (ii), (iii) that $A^{(j)}$, $j=1,2,3$, and $q$ can be expanded into power series 
\begin{equation}
\label{eq_8_3}
A^{(j)}(x,z)=\sum_{k=1}^\infty A^{(j),k}(x)\frac{z^k}{k!}, \quad A^{(j),k}(x):=\p_z^kA^{(j)}(x,0)\in C^{0,\alpha}(\overline{\Omega};  S^j),
\end{equation}
converging in the $C^{0,\alpha}(\overline{\Omega};  S^j)$ topology, and 
\begin{equation}
\label{eq_8_4}
q(x,z)=\sum_{k=2}^\infty q^{k}(x)\frac{z^k}{k!}, \quad q^{k}(x):=\p_z^kq (x,0)\in C^{0,\alpha}(\overline{\Omega}),
\end{equation}
converging in the $C^{0,\alpha}(\overline{\Omega})$ topology.  

In what follows, when $X,Y$ are Banach spaces, we equip $X\times Y$ with the norm $\|(x,y)\|_{X\times Y}=\|x\|_{X}+\|y\|_{Y}$. Then $X\times Y$ becomes a Banach space. We have the following result. 
\begin{thm}
\label{thm_well-posedness}
Assume that $A^{(j)}: \overline{\Omega}\times \C\to S^j$, $j=1,2,3$, and $q:\overline{\Omega}\times\C\to \C$ satisfy the conditions  (i), (ii), and (iii). Then there exist $\delta>0$ and $C>0$ such that for any $(f,g)\in B_\delta(\p \Omega):=\{(f,g)\in C^{4,\alpha}(\p \Omega)\times C^{3,\alpha}(\p \Omega): \|f\|_{C^{4,\alpha}(\p \Omega)}+ \|g\|_{C^{3,\alpha}(\p \Omega)}<\delta\}$, the problem \eqref{eq_8_2} has a  solution $u=u_{f,g}\in C^{4,\alpha}(\overline{\Omega})$ satisfying 
\[
\|u\|_{C^{4,\alpha}(\overline{\Omega})}\le C(\|f\|_{C^{4,\alpha}(\p \Omega)}+\|g\|_{C^{3,\alpha}(\p \Omega)}).
\]
The solution is unique within the class $\{u\in C^{4,\alpha}(\overline{\Omega}): \|u\|_{C^{4,\alpha}(\overline{\Omega})}< C\delta \}$, and it depends holomorphically on $(f,g)\in B_\delta(\p \Omega)$. Furthermore, the map 
\[
B_\delta(\p \Omega)\to  C^{2,\alpha}(\p \Omega)\times C^{1,\alpha}(\p \Omega), \quad   (f,g)\mapsto (\p_\nu^2 u|_{\p \Omega}, \p_\nu^3 u|_{\p \Omega})
\]
is holomorphic. 
\end{thm}

\begin{proof}
We shall follow \cite[Appendix A]{Krup_Uhlmann_2020}, see also \cite{Lassas_Liimatainen_Lin_Salo_2021}, and rely on the implicit function theorem for holomorphic maps between complex Banach spaces, see \cite[p. 144]{Poschel_Trubowitz_1987}. In doing so, we consider the map 
\begin{equation}
\label{eq_8_5}
\begin{aligned}
F: C^{4,\alpha}(\p \Omega)\times C^{3,\alpha}(\p \Omega)\times C^{4,\alpha}(\overline{\Omega})\to C^{0,\alpha}(\overline{\Omega})\times C^{4,\alpha}(\p \Omega)\times C^{3,\alpha}(\p \Omega),\\
F(f,g,u)=\big(L_{A^{(1)}, A^{(2)}, A^{(3)},q}u, u|_{\p \Omega}-f, \p_\nu u|_{\p \Omega}-g\big).
\end{aligned}
\end{equation}
Let us check that $F$ enjoys the mapping properties \eqref{eq_8_5}. To that end let $u\in C^{4,\alpha}(\overline{\Omega})$ and first note that $(-\Delta)^2u\in C^{0,\alpha}(\overline{\Omega})$. Let $j=1,2,3$ and let us now check that $A^{(j)}(\cdot, u(\cdot))\in C^{0,\alpha}(\overline{\Omega}; S^j)$. As in  \cite{Krup_Uhlmann_2020}, using Cauchy's estimates, we see that the coefficients $A^{(j),k}$ in \eqref{eq_8_3} satisfy
\begin{equation}
\label{eq_8_6}
\|A^{(j),k}\|_{C^{0,\alpha}(\overline{\Omega}; S^j)}\le \frac{k!}{R^k}\sup_{|z|=R}\|A^{(j)}(\cdot, z)\|_{C^{0,\alpha}(\overline{\Omega}; S^j)},
\end{equation}
where $R>0$ and $k=1,2,\dots$. It follows from \eqref{eq_8_6} and \eqref{eq_8_1} that 
\begin{equation}
\label{eq_8_7}
\bigg\|A^{(j),k} \frac{u^k}{k!}\bigg\|_{C^{0,\alpha}(\overline{\Omega}; S^j)}\le \frac{C^k}{R^k}\|u\|^k_{C^{0,\alpha}(\overline{\Omega})}\sup_{|z|=R}\|A^{(j)}(\cdot, z)\|_{C^{0,\alpha}(\overline{\Omega}; S^j)}.
\end{equation}
Choosing $R=2C\|u\|_{C^{0,\alpha}(\overline{\Omega})}$, we see from \eqref{eq_8_7} that the series 
$\sum_{k=1}^\infty A^{(j),k}(x)\frac{u^k}{k!}$ converges in $C^{0,\alpha}(\overline{\Omega}; S^j)$, and hence, $A^{(j)}(\cdot, u(\cdot))\in C^{0,\alpha}(\overline{\Omega}; S^j)$, $j=1,2,3$. Similarly, we see that $q(\cdot, u(\cdot))\in C^{0,\alpha}(\overline{\Omega})$. Thus,  we conclude from \eqref{int_eq_1} that $L_{A^{(1)},A^{(2)},A^{(3)},q}u\in C^{0,\alpha}(\overline{\Omega})$.   

Let us next show that $F$, given by \eqref{eq_8_5}, is holomorphic. First, $F$ is locally bounded as it is continuous in $(f,g,u)$. Therefore, it suffices to check that $F$ is weakly holomorphic, see \cite[p. 133]{Poschel_Trubowitz_1987}. In doing so, letting $(f_1,g_1,u_1), (f_2,g_2,u_2)\in C^{4,\alpha}(\p \Omega)\times C^{3,\alpha}(\p \Omega)\times C^{4,\alpha}(\overline{\Omega})$, we have to check that the map 
\[
\lambda\mapsto F((f_1,g_1,u_1)+\lambda(f_2,g_2,u_2))
\]
is holomorphic in $\C$ with values in $C^{0,\alpha}(\overline{\Omega})\times C^{4,\alpha}(\p \Omega)\times C^{3,\alpha}(\p \Omega)$. To that end, it suffices to show that the map $\lambda\mapsto A^{(j)}(\cdot, u_1(\cdot)+\lambda u_2(\cdot))$ is holomorphic in $\C$ with values in $C^{0,\alpha}(\overline{\Omega}; S^j)$, $j=1,2,3$. The latter follows from the fact that the series 
 $\sum_{k=1}^\infty \frac{A^{(j),k}}{k!}(u_1+\lambda u_2)^k$ converges in $C^{0,\alpha}(\overline{\Omega}; S^j)$, locally uniformly in $\lambda\in \C$, see \eqref{eq_8_7}. 

Now we have $F(0,0,0)=0$ and the partial differential 
\begin{equation}
\label{eq_8_8}
\p_uF(0,0,0): C^{4,\alpha}(\overline{\Omega})\to C^{0,\alpha}(\overline{\Omega})\times C^{4,\alpha}(\p \Omega)\times C^{3,\alpha}(\p \Omega)
\end{equation}
is given by
\[
\p_uF(0,0,0) v=((-\Delta)^2 v, v|_{\p \Omega}, \p_\nu v|_{\p \Omega}). 
\]
Thus, it follows from \cite[Theorem 2.19]{Gazzola_Grunau_Sweers_book_2010}, see also \cite[Section 6.6]{Evans_book} , that the map \eqref{eq_8_8} is a linear isomorphism.

An application of the implicit function theorem, see  \cite[p. 144]{Poschel_Trubowitz_1987}, gives that there exists $\delta>0$ and a unique holomorphic map $S: B_\delta(\p \Omega)\to C^{4,\alpha}(\overline{ \Omega})$ such that $S(0,0)=0$ and $F(f,g,S(f,g))=0$ for all $(f,g)\in B_\delta(\p \Omega)$.  Setting $u=S(f,g)$ and noting that $S$ is Lipschitz continuous with $S(0,0)=0$,  we get 
\[
\|u\|_{C^{4,\alpha}(\overline{\Omega})}\le C(\|f\|_{C^{4,\alpha}(\p \Omega)}+\|g\|_{C^{3,\alpha}(\p \Omega)}).
\]
\end{proof}

\section{Proofs of injectivity results for generalized momentum ray transforms}
\label{sec_app_transforms}

The purpose of this appendix is to provide proofs of key injectivity results for generalized momentum ray transforms required to establish Theorem \ref{thm_density}. We will specifically present the proofs for Theorem \ref{thm_Bhattacharyya_Krishnan_Sahoo}, Lemma \ref{lem_inversion_1}, and Lemma \ref{lem_inversion_2}. These results are established in \cite[Theorem 4.18]{Bhattacharyya_Krishnan_Sahoo_2021}, \cite[Lemma 3.7]{Bhattacharyya_Krishnan_Sahoo_2021}, and \cite[Lemma 6.7]{Sahoo_Salo_2023}, considering tensor fields of arbitrary order $m\in \N$. In this appendix, we shall present the proofs given in those works, making them more direct when tailored to the case of $m$--tensor fields where $m=0,1,2, 3$, as needed in this paper. We will also give special consideration to continuous tensor fields with compact support. This appendix is included for the completeness and convenience of the reader only.

\subsection{Proof of Theorem \ref{thm_Bhattacharyya_Krishnan_Sahoo}}

\label{sec_app_transforms_1}
We shall follow the proof of  \cite[Theorem 4.18]{Bhattacharyya_Krishnan_Sahoo_2021}.  Let first $m=3$ and let us assume that $J^{3,3}F(x,\xi)=0$, for all $(x,\xi)\in \R^n\times \mathbb{S}^{n-1}$. Then  $J^{3,3}F(x,-\xi)=0$, for all $(x,\xi)\in \R^n\times \mathbb{S}^{n-1}$. Using that $J^{3,3}F(x,\xi)\pm J^{3,3}F(x,-\xi) =0$ and \eqref{eq_201_1}, we get 
\begin{equation}
\label{eq_201_3}
J^{1,3}f^{(1)}(x,\xi)+ J^{3,3}f^{(3)}(x,\xi)=0, \quad J^{0,3}f^{(0)}(x,\xi)+ J^{2,3}f^{(2)}(x,\xi)=0,
\end{equation}
for all $(x,\xi)\in \R^n\times \mathbb{S}^{n-1}$. Using \eqref{eq_201_2}, we obtain from \eqref{eq_201_3} that 
\begin{equation}
\label{eq_201_4}
J^{3,3}(i_\delta f^{(1)}+f^{(3)})(x,\xi)=0,
\end{equation}
\begin{equation}
\label{eq_201_5}
J^{2,3}(i_\delta f^{(0)}+f^{(2)})(x,\xi)=0,
\end{equation}
for all $(x,\xi)\in \R^n\times \mathbb{S}^{n-1}$.  Using \eqref{eq_201_0}, we get from \eqref{eq_201_4} that 
$I^{3,3}(i_\delta f^{(1)}+f^{(3)})(x,\xi)=0$,
for all $(x,\xi)\in \R^n\times (\R^n\setminus\{0\})$, and therefore, in $\mathcal{D}'(\R^n\times \mathbb{S}^{n-1})$. Hence, Theorem \ref{thm_higher_order_on} implies that  $i_\delta f^{(1)}+f^{(3)}=0$. 

Applying the operator $\xi\cdot\p_x$ to \eqref{eq_201_5} in the dense of distributions, and using \eqref{eq_6_4}, we obtain that $J^{2,2}(i_\delta f^{(0)}+f^{(2)})=0$ in $\mathcal{D}'(\R^n\times\mathbb{S}^{n-1})$. As $J^{2,2}(i_\delta f^{(0)}+f^{(2)})\in C(\R^n\times\mathbb{S}^{n-1})$, we have $J^{2,2}(i_\delta f^{(0)}+f^{(2)})(x,\xi)=0$ for all $(x,\xi)\in \R^n\times\mathbb{S}^{n-1}$. In view of   \eqref{eq_201_0}, we get $I^{2,2}(i_\delta f^{(0)}+f^{(2)})(x,\xi)=0$ for all $(x,\xi)\in \R^n\times (\R^n\setminus\{0\})$. Hence, Theorem \ref{thm_higher_order_on} implies that  $i_\delta f^{(0)}+f^{(2)}=0$. This establishes the implication in (i) in one direction. The proof of the converse follows by a direct computation.  The proof of (ii)-(iv) is similar. This completes the proof of Theorem \ref{thm_Bhattacharyya_Krishnan_Sahoo}.

\subsection{Proof of Lemma \ref{lem_inversion_1}}

\label{sec_app_transforms_2}
We shall follow the proof of \cite[Lemma 3.7]{Bhattacharyya_Krishnan_Sahoo_2021}. 
First, assume that $m=3$. We have the following decomposition,  
\begin{equation}
\label{eq_5_5}
f^{(3)}=g^{(3)}+i_\delta v^{(1)}, 
\end{equation}
where $v^{(1)}\in C_0(\R^{n-1}; S^{1}(\R^n))$ and $g^{(3)}\in C_0(\R^{n-1}; S^{3}(\R^n))$ is such that  
\begin{equation}
\label{eq_5_5_1}
j_\delta g^{(3)}=0,
\end{equation}
see \cite[Lemma 2.3]{Dairbekov_Sharafutdinov_2011}. Indeed, applying $j_\delta$ to \eqref{eq_5_5} and using \eqref{eq_5_5_1} and \eqref{eq_5_3_1}, we get $v^{(1)}=\frac{3}{n+2}j_\delta f^{(3)}\in C_0(\R^{n-1}; S^{1}(\R^n))$, and 
therefore, $g^{(3)}=f^{(3)}-\frac{3}{n+2}i_\delta j_\delta f^{(3)}\in C_0(\R^{n-1}; S^{3}(\R^n))$. The condition \eqref{eq_5_5_1} follows.

A direct computation using \eqref{eq_5_3}, \eqref{eq_5_2}, and the fact that $(e_1+i\eta)\cdot (e_1+i\eta)=0$ gives that $i_\delta v^{(1)}$ is in the kernel of  \eqref{eq_5_4}. Thus, substituting \eqref{eq_5_5} in \eqref{eq_5_4}, we obtain that
\begin{equation}
\label{eq_5_6}
\sum_{i_1,i_2, i_3=1}^n \int_{\R} t^3 g^{(3)}_{i_1 i_2 i_3}( y_0'+t\eta')(e_1+i\eta)_{i_1}(e_1+i\eta)_{i_2}(e_1+i\eta)_{i_3}dt=0,
\end{equation}
for  all $y_0'\in \R^{n-1}$ and all  $\eta\in \R^n$  such that $e_1\cdot\eta =0$ and $|\eta|=1$.

We shall proceed to show that $g^{(3)}=0$. First,  a direct computation using that $\eta_{1}=0$ gives 
\begin{equation}
\label{eq_5_7}
\begin{aligned}
\sum_{i_1,i_2, i_3=1}^n &g^{(3)}_{i_1 i_2 i_3}(e_1+i\eta)_{i_1}(e_1+i\eta)_{i_2}(e_1+i\eta)_{i_3}
=g^{(3)}_{111}+3 i \sum_{i_1=2}^n g^{(3)}_{i_1 11}\eta_{i_1}\\
 &+ 3i^2 \sum_{i_1,i_2=2}^n g^{(3)}_{i_1i_2 1}\eta_{i_1}\eta_{i_2}+i^3 \sum_{i_1,i_2, i_3=2}^n g^{(3)}_{i_1 i_2 i_3}\eta_{i_1}\eta_{i_2}\eta_{i_3}
=g^{(3)}_{111}+3 i \sum_{i_1=1}^{n-1} g^{(3)}_{i_1+1 11}\eta'_{i_1}\\
 &+ 3i^2 \sum_{i_1,i_2=1}^{n-1} g^{(3)}_{i_1+1i_2+1 1}\eta'_{i_1}\eta'_{i_2}+i^3 \sum_{i_1,i_2, i_3=1}^{n-1} g^{(3)}_{i_1+1 i_2+1 i_3+1}\eta'_{i_1}\eta'_{i_2}\eta'_{i_3}.
\end{aligned}
\end{equation}
Substituting \eqref{eq_5_7} in \eqref{eq_5_6}, we get that 
\begin{equation}
\label{eq_5_8}
\begin{aligned}
 \int_{\R} t^3\bigg( g^{(3)}_{111}( y_0'+t\eta') &+3 i \sum_{i_1=1}^{n-1} g^{(3)}_{i_1+1 11}( y_0'+t\eta')\eta'_{i_1}\\
&+ 3i^2 \sum_{i_1,i_2=1}^{n-1} g^{(3)}_{i_1+1i_2+1 1}( y_0'+t\eta')\eta'_{i_1}\eta'_{i_2}\\
 &+i^3 \sum_{i_1,i_2, i_3=1}^{n-1} g^{(3)}_{i_1+1 i_2+1 i_3+1}(y_0'+t\eta')\eta'_{i_1}\eta'_{i_2}\eta'_{i_3}\bigg)dt=0,
\end{aligned}
\end{equation}
for all $(y_0', \eta') \in \R^{n-1}\times \mathbb{S}^{n-2}$.   An application of Theorem \ref{thm_Bhattacharyya_Krishnan_Sahoo}, in the case of tensor fields in $C_0(\R^{n-1}; \mathcal{S}^3(\R^{n-1}))$, to \eqref{eq_5_8}, and a subsequent use of  \eqref{eq_5_3}, \eqref{eq_5_2}  give that for all $i_1,i_2, i_3=1,\dots, n-1$, 
\begin{equation}
\label{eq_5_9}
3i^2 g^{(3)}_{i_1+1i_2+1 1}= -(i_\delta g^{(3)}_{111})_{i_1i_2}=-\frac{1}{2} g^{(3)}_{111} (\delta_{i_1i_2}+\delta_{i_2i_1}),
\end{equation}
and
\begin{equation}
\label{eq_5_10}
i^3 g^{(3)}_{i_1+1 i_2+1 i_3+1}=-(i_\delta \tilde g^{(1)})_{i_1i_2i_3}, \quad (\tilde g^{(1)})_{i_1} =3i g^{(3)}_{i_1+1 11}.
\end{equation}
Now \eqref{eq_5_9} implies that 
\begin{equation}
\label{eq_5_11}
g_{j_1j_21}^{(3)}=0, \quad j_1\ne j_2, \quad j_1,j_2=2,\dots, n,
\end{equation}
and 
\begin{equation}
\label{eq_5_12}
3g^{(3)}_{j_1j_1 1}=g^{(3)}_{111}, \quad  j_1=2,\dots, n.
\end{equation}
It follows from \eqref{eq_5_5_1} that for $ j_1=1,\dots, n$,
\begin{equation}
\label{eq_5_13}
(j_\delta g^{(3)})_{j_1}=\sum_{l=1}^n g^{(3)}_{j_1 ll}= g^{(3)}_{j_1 11}+ \sum_{l=2}^n  g^{(3)}_{j_1 ll}=0.
\end{equation}
Now \eqref{eq_5_13} with $j_1=1$ and \eqref{eq_5_12} imply that 
\begin{equation}
\label{eq_5_14}
g^{(3)}_{j_1j_1 1}=0, \quad j_1=1,2,\dots, n. 
\end{equation}

It follows from \eqref{eq_5_10} that 
\begin{equation}
\label{eq_5_15}
g^{(3)}_{j_1j_2 j_3}=0, \quad j_1\ne j_2\ne j_3, \quad j_1,j_2,j_3=2,\dots, n,
\end{equation}
\begin{equation}
\label{eq_5_16}
g^{(3)}_{j_1j_1 j_1}=3 g^{(3)}_{j_1 11}, \quad  j_1=2,\dots, n,
\end{equation}
and 
\begin{equation}
\label{eq_5_17}
g_{j_1j_1 j_3}^{(3)}=g^{(3)}_{j_3 11}, \quad j_1\ne j_3,\quad  j_1, j_3=2,\dots, n.
\end{equation}
Combining \eqref{eq_5_13} with $j_1=j_3$ and \eqref{eq_5_17}, we get 
\begin{equation}
\label{eq_5_18}
\begin{aligned}
&g^{(3)}_{j_3 11}=0, \quad j_3=2,\dots, n,\\
&g^{(3)}_{j_1j_1 j_3}=0, \quad j_1\ne j_3,\quad  j_1, j_3=2,\dots, n.
\end{aligned}
\end{equation}
Thus, in view of \eqref{eq_5_18}, \eqref{eq_5_16} gives that 
\begin{equation}
\label{eq_5_19}
g^{(3)}_{j_1j_1 j_1}=0, \quad  j_1=2,\dots, n.
\end{equation}
We conclude from \eqref{eq_5_11},  \eqref{eq_5_14},  \eqref{eq_5_15}, \eqref{eq_5_18}, and \eqref{eq_5_19} that $g^{(3)}=0$. This completes the proof in the case $m=3$. 

Assume now that $m=2$. Similar to the case $m=3$, we have the following  decomposition 
$f^{(2)}=g^{(2)}+i_\delta v^{(0)}$, where $v^{(0)}= \frac{1}{n}j_\delta f^{(2)} \in C_0(\R^{n-1}, S^{0}(\R^n))$ and $g^{(2)}= f^{(2)}-\frac{1}{n}i_\delta j_\delta f^{(2)}\in  C_0(\R^{n-1}, S^{2}(\R^n))$ is such that 
\begin{equation}
\label{eq_5_5_1_tensor_2}
j_\delta g^{(2)}=0.
\end{equation}
We reduce the proof of the lemma to the following result:  if 
\begin{equation}
\label{eq_5_20}
\begin{aligned}
\int_{\R} t^2\bigg( g^{(2)}_{11}( y_0'+t\eta') &+2 i \sum_{i_1=1}^{n-1} g^{(2)}_{i_1+1 1}(y_0'+t\eta')\eta'_{i_1}\\
 &+ i^2 \sum_{i_1,i_2=1}^{n-1} g^{(2)}_{i_1+1i_2+1 }(y_0'+t\eta')\eta'_{i_1}\eta'_{i_2} \bigg)dt=0,
\end{aligned}
\end{equation}
for  all $(y_0', \eta')\in \R^{n-1}\times\mathbb{S}^{n-2}$,  show that $g^{(2)}=0$.  An application of Theorem \ref{thm_Bhattacharyya_Krishnan_Sahoo}, in the case of  $C_0(\R^{n-1}, \mathcal{S}^2(\R^{n-1}))$, to \eqref{eq_5_20}, and a subsequent use of  \eqref{eq_5_3}, \eqref{eq_5_2}  give that
\begin{equation}
\label{eq_5_21}
g^{(2)}_{j_1 1}=0, \quad j_1=2,\dots, n,
\end{equation}
\begin{equation}
\label{eq_5_22}
g^{(2)}_{j_1 j_2}=0, \quad j_1\ne j_2, \quad j_1, j_2=2,\dots, n,
\end{equation}
and 
\begin{equation}
\label{eq_5_23}
g^{(2)}_{j_1 j_1}=g^{(2)}_{11}, \quad  j_1=2,\dots, n.
\end{equation}
Now \eqref{eq_5_5_1_tensor_2} implies that 
\[
j_\delta g^{(2)}=g^{(2)}_{11}+\sum_{k=2}^n g^{(2)}_{kk}=0,
\]
and therefore, \eqref{eq_5_23} gives that 
\begin{equation}
\label{eq_5_24}
g^{(2)}_{j_1 j_1}=0, \quad  j_1=1,\dots, n.
\end{equation}
We conclude from \eqref{eq_5_21}, \eqref{eq_5_22}, and \eqref{eq_5_24} that $g^{(2)}=0$. This completes the proof of Lemma in the case $m=2$. 

Let $m=1$. Then \eqref{eq_5_4} can be written as 
\[
 \int_{\R} t \big(f^{(1)}_1( y_0'+t\eta') + \sum_{i_1=1}^{n-1} i f^{(1)}_{i_1+1}( y_0'+t\eta')\eta'_{i_1}\big) dt=0,
\]
for all  $(y_0', \eta') \in \R^{n-1}\times \mathbb{S}^{n-2}$. An application of   Theorem \ref{thm_Bhattacharyya_Krishnan_Sahoo} gives $f^{(1)}=0$. The proof of Lemma \ref{lem_inversion_1} is complete. 

\subsection{Proof of Lemma \ref{lem_inversion_2}}
\label{sec_app_transforms_3}

 First, we write 
\begin{equation}
\label{eq_5_26}
f^{(2)}=g^{(2)}+i_\delta v^{(0)},
\end{equation}
where 
\begin{align*}
v^{(0)}= \frac{1}{n}j_\delta f^{(2)} \in C_0(\R^{n-1}, S^{0}(\R^n)), \quad g^{(2)}= f^{(2)}-\frac{1}{n}i_\delta j_\delta f^{(2)}\in  C_0(\R^{n-1}, S^{2}(\R^n)),
\end{align*}
 so  that 
$j_\delta g^{(2)}=0$, see \cite[Lemma 2.3]{Dairbekov_Sharafutdinov_2011}.  Substituting \eqref{eq_5_26} in \eqref{eq_5_25}, and using  that $\eta_{1}=0$, we get 
\begin{equation}
\label{eq_5_27}
\begin{aligned}
\int_{\R} t^{2+k}\bigg(g^{(2)}_{11}(y_0'+t\eta')+&2i\sum_{i_1=1}^{n-1} g^{(2)}_{i_1+1 1}(y_0'+t\eta')\eta'_{i_1} \\
&+ i^2\sum_{i_1, i_2=1}^{n-1} g^{(2)}_{i_1+1 i_2+1}(y_0'+t\eta')\eta'_{i_1}\eta'_{i_2} \bigg)dt\\
&+ \int_{\R} t^{1+k}\bigg(f^{(1)}_{1}(y_0'+t\eta')+i\sum_{i_1=1}^{n-1} f^{(1)}_{i_1+1}(y_0'+t\eta')\eta'_{i_1} \bigg)dt=0,
\end{aligned}
\end{equation}
for all $\eta'\in \R^{n-1}$, $|\eta'|=1$, all $y_0'\in \R^{n-1}$, and $k=0,1$. 
Now \eqref{eq_5_27} can be written as follows, 
\begin{equation}
\label{eq_5_28}
J^{2,2+k} G(y_0',\eta')+ J^{1,k+1}F(y_0',\eta')=0,
\end{equation}
for all $(y_0',\eta')\in \R^{n-1}\times\mathbb{S}^{n-2}$. Here $G=\tilde g^{(0)}+\tilde g^{(1)}+\tilde g^{(2)}\in C_0(\R^{n-1},\mathcal{S}^2(\R^{n-1}))$, and  
$F=\tilde f^{(0)}+\tilde f^{(1)}\in C_0(\R^{n-1},\mathcal{S}^1(\R^{n-1}))$, where 
\begin{align*}
&\tilde g^{(0)}=g^{(2)}_{11}, \quad \tilde g^{(1)}_{i_1}= 2i g^{(2)}_{i_1+1 1}, \quad \tilde g^{(2)}_{i_1 i_2}= i^2g^{(2)}_{i_1+1 i_2+1},\\
&\tilde f^{(0)}= f^{(1)}_{1}, \quad \tilde f^{(1)}_{i_1}=i f^{(1)}_{i_1+1}, \quad 1\le i_1,i_2\le n-1. 
\end{align*}
Applying $\eta'\cdot\p_{y_0'}$ to \eqref{eq_5_28} in the sense of distributions, and using \eqref{eq_6_4}, we get 
\begin{equation}
\label{eq_5_29}
(k+2 )J^{2,1+k} G+ (k+1)J^{1,k}F=0 \quad \text{in}\quad \mathcal{D}'(\R^{n-1}\times\mathbb{S}^{n-2}). 
\end{equation}
Now since the left hand side of \eqref{eq_5_29} belongs to $C(\R^{n-1}\times \mathbb{S}^{n-2})$, \eqref{eq_5_29} implies that 
\begin{equation}
\label{eq_5_29_new_cont}
(k+2 )J^{2,1+k} G(y_0',\eta')+ (k+1)J^{1,k}F(y_0',\eta')=0, 
\end{equation}
for all $(y_0',\eta')\in \R^{n-1}\times\mathbb{S}^{n-2}$. 
Combining \eqref{eq_5_28} with $k=0$ and \eqref{eq_5_29_new_cont} with $k=1$, we obtain that 
\[
J^{2,2}G(y_0',\eta')=0, \quad J^{1,1}F(y_0',\eta')=0,
\]
for all $(y_0',\eta')\in \R^{n-1}\times\mathbb{S}^{n-2}$. 
Thus, Theorem \ref{thm_Bhattacharyya_Krishnan_Sahoo} implies that $f^{(1)}=0$. Now the equality $J^{2,2}G(y_0',\eta')=0$ is precisely the equality \eqref{eq_5_20} in the proof of Lemma \ref{lem_inversion_1}. Hence, in view of the fact that $j_\delta g^{(2)}=0$, we conclude that $g^{(2)}=0$. The proof of Lemma \ref{lem_inversion_2} is complete.

\end{appendix}

\section*{Acknowledgements}
The research of S.B. is partially supported by the grant: SRG/2022/00129 from the Science and Engineering Research Board, India.
The research of K.K. is partially supported by the National Science Foundation (DMS 2109199). The research of S. K. S is partly supported by the Academy of Finland (Centre of Excellence in Inverse Modelling and Imaging, grant 284715) and by the European Research Council under Horizon 2020 (ERC CoG 770924). The research of G.U. is partially supported by NSF, a Robert R. and Elaine F. Phelps Endowed Professorship at the University of Washington, a Si-Yuan Professorship at IAS, HKUST.

\end{document}